\documentclass{amsart}

\usepackage{Caetano_Dubois_Graille}

\begin{document}

\title[A result of convergence for a D1Q2 lattice Boltzmann scheme]{A result of convergence for a mono-dimensional two-velocities lattice Boltzmann scheme}

\author[F.~Caetano]{Filipa~Caetano$^1$}
\author[F.~Dubois]{Fran\c cois~Dubois$^{1,2}$}
\author[B.~Graille]{Benjamin~Graille$^{1, \natural}$}

\address{$^\natural$ corresponding author: \textup{benjamin.graille@universite-paris-saclay.fr}}
\address{$^1$ Universit\'e Paris-Saclay, CNRS, Laboratoire de math\'ematiques d'Orsay, 91405, Orsay, France.}
\address{$^2$ Conservatoire National des Arts et M\'etiers, LMSSC laboratory, Paris, France}

\date{\today}


\begin{abstract}
  We consider a mono-dimensional two-velocities scheme used to approximate the solutions of a scalar hyperbolic conservative partial differential equation. We prove the convergence of the discrete solution towards the unique entropy solution by first estimating the supremum norm and the total variation of the discrete solution, and second by constructing a discrete kinetic entropy-entropy flux pair being given a continuous entropy-entropy flux pair of the hyperbolic system. We finally illustrate our results with numerical simulations of the advection equation and the Burgers equation.
\end{abstract}

\maketitle


\section{Introduction}

The lattice Boltzmann method is a numerical method which is largely used to simulate fluid dynamics equations, such as Navier Stokes, heat, acoustics equations, multi-phase and multi-component fluids (see Succi \cite{Suc:2015:0}, Lallemand and Luo \cite{LalLuo:2000:0}, and \textsl{e.g.} \cite{kru:2017:0}). Its origin is in a discretized velocities version of the continuous Boltzmann equation (see Broadwell \cite{Bro:1964:0} and Gatignol \cite{Gat:1975:0}), with a specific collision kernel. The algorithm of the lattice Boltzmann method reads then as a fully discretized Boltzmann equation on a lattice.  

The lattice Boltzmann method supposes that particles progress on a discrete Cartesian lattice with a finite set of speeds. In one time step, each velocity allows particles to jump from one vertex of the lattice to another one. One iteration of the method can be described in two steps: a relaxation step, which is local to each vertex and which corresponds to the collision of the particles, followed by a transport step, which corresponds to the evolution of the particles on the lattice.

Despite the fact that this method is widely used, the mathematical numerical analysis is far from being complete. In \cite{Rhe:2008:0}, a stability analysis is proposed by Rheinl\"ander for a two-velocities lattice Boltzmann scheme for the linear advection equation. In \cite{JunYan:2009:0}, Junk and Yang studied the convergence of approximation of smooth solutions for the incompressible Navier-Stokes equations. Concerning the linear one-dimensional convection-diffusion equation, Dellacherie proves the convergence in \(\Linf\) norm of the solutions of a two-velocities lattice Boltzmann scheme \cite{Dell:2014:0}. We remark that in previous works studying lattice Boltzmann methods, links between lattice Boltzmann and finite difference methods were done (see Junk \cite{Jun:2001:0} and \cite{Gra201400}). Concerning non-linear models, in \cite{BogLovYep:2004:0}, Boghosian, Love, and Yepez developed a two-velocities entropic scheme for the viscous Burgers equation and studied its properties.

In this contribution we consider a two-velocities lattice Boltzmann scheme in one dimension, known in the domain as \duqd, in order to approach a scalar conservation law. The purpose of our paper is to prove the convergence of this lattice Boltzmann scheme towards the unique entropy solution of the hyperbolic equation as the mesh size tends to zero, when the relaxation parameter of the scheme lies between 0 and~1. 

For this purpose, we make a link with the relaxation system of Jin and Xin \cite{JinXin:1995:0} and we use, in the context of our scheme, some techniques used to study the convergence of relaxation schemes based on the Jin and Xin approximation. Nevertheless, we point out that the lattice Boltzmann scheme is not constructed as a relaxation scheme, and it does not correspond to a natural discretization of the Jin and Xin system.

The convergence of the solutions of the Jin and Xin system towards the equilibrium, described by the scalar conservation law, was obtained independently by Natalini \cite{Nat:1996:0} and by Serre \cite{Ser:2000:0}. Concerning the numerical approximation of the scalar conservation law, Aregba-Driollet and Natalini in \cite{AreNat:1996:0} and \cite{Nat:1998:1}, and Lattanzio and Serre in \cite{LatSer:2001:0}, proved the convergence of finite volumes relaxation schemes based on the relaxation approximation of Jin and Xin. The relaxation schemes investigated in these works depend on two parameters (the relaxation rate $\jinxinparam$ and the mesh step $\dt$) and the convergence of the schemes is proven by passing to the limit both parameters. In \cite{AreNat:1996:0,Nat:1998:1}, uniform \(\Linf\) and total variation estimates are obtained to pass to the limit separately in both parameters. In \cite{LatSer:2001:0}, compensated compactness techniques are used to study the limit of a relaxation scheme as both parameters go simultaneously to 0.

The lattice Boltzmann scheme we investigate only involves one parameter $\dt$. Although this scheme is not a relaxation scheme, it can be linked to this kind of schemes, with a particular choice of the relaxation rate $\jinxinparam$ proportional to $\dt$. As we can see in the figure~\ref{fig:lbmvsrelax}, the convergence of the LB scheme corresponds to a diagonal limit (a limit on $\jinxinparam$ and on $\Delta t$ at the same time).

\begin{figure}
  \centering
  \begin{tikzpicture}[scale=1.5]
      \coordinate (O) at (0,0) ;
      \coordinate (A) at (2,0) ;
      \coordinate (B) at (0,2) ;
      \coordinate (Ca) at (1.8,1.8) ;
      \coordinate (Cb) at (1.3,1.8) ;
      \coordinate (D) at (0,1.8) ;
      \coordinate (E) at (1.3,0) ;
      \draw[thin, dashed, ->] (O) -- (A) node[right] {$\jinxinparam$};
      \draw[thin, dashed, ->] (O) -- (B) node[above] {$\Delta t$};
      \draw[thick, orange] (O) -- (Ca) node[above, right] {\footnotesize lattice Boltzmann methods: $\jinxinparam$ and $\Delta t$ proportional};
      \draw[thick, blue] (O) -- (D) node[midway, left] {\footnotesize finite volume: $\jinxinparam=0$};
      \draw[thick, violet] (Cb) -- (E) node[midway, right] {\footnotesize relaxation scheme: $\jinxinparam$ fixed};
      \node[below] (O) {$0$};
      \node[left] (O) {$0$};
  \end{tikzpicture}
  \caption{Representation of the schemes parameters in the plane $(\jinxinparam, \dt$).}
  \label{fig:lbmvsrelax}
\end{figure}

The aim of our work is to prove, by using finite volumes methods techniques  such as total variation and \(\Linf\) bounds inspired by \cite{AreNat:1996:0}, \cite{Nat:1998:1}, and \cite{LatSer:2001:0}, that the solutions of the \duqd scheme converge towards the numerical solution of a scalar hyperbolic conservation law, under a stability condition. To the knowledge of the authors, this result is the first one of convergence for the solution of a lattice Boltzmann scheme towards the solution of a nonlinear conservation equation. 

The plan of the paper is as follows. In section~\ref{sec:desc_scheme} we describe the \duqd scheme and recall its construction, by following the framework of d'Humi\`eres \cite{dHu:1994:0}. In section~\ref{sec:TVLinf_est} we establish \(\Linf\) and total variation discrete estimates. By using these estimates, in section~\ref{sec:conv_scheme} we prove the convergence of the \duqd scheme towards a weak solution of the scalar conservation law, in the case where the relaxation parameter of the scheme lies between 0 and 1. In section~\ref{sec:num_ent}, by following Serre \cite{Ser:2000:0}, Bouchut \cite{Bou:2003:0}, and \cite{Dub:2013:0}, we introduce numerical entropies for the \duqd scheme. Based on a \(\Lun\) estimation of the equilibrium gap independent of the mesh step established in section~\ref{sec:TVLinf_est}, we prove in section~\ref{sec:conv_entropysol} the convergence of the \duqd scheme towards the unique entropy solution of the scalar conservation law. Finally, section~\ref{sec:num_ill} presents several numerical tests that illustrate the convergence of the scheme. We also present some tests that precise the numerical convergence rate of the scheme, however without theoretical results.

\section{Description of the \duqd scheme}
\label{sec:desc_scheme}

We consider the following mono-dimen\-sional scalar conservation law 
\begin{equation}\label{eq:s1_edp}
 \drondt \incu(\vart,\varx) + \drondx \phi(\incu)(\vart,\varx) = 0, 
 \qquad \vart>0, \ \varx\in\R,
\end{equation}
where the flux \(\phi\) is a \(\Cun\) function on \(\R\), with the initial condition
\begin{equation}\label{eq:s1_ic}
 \incu(0,\varx) = \incu^0(\varx), \qquad \varx\in\R.
\end{equation}
It is well known that the Cauchy problem 
(\ref{eq:s1_edp}-\ref{eq:s1_ic}) 
possesses a unique entropy solution which belongs to 
\(\Linf(]0,T[\times\R)\), for all \(T>0\), and such that \(\incu(\vart,\cdot)\in\BV(\R)\), \(\forall\vart>0\), provided that the initial data \(\incu^0\in\Linf(\R)\cap\BV(\R)\) (see \textsl{e.g.} \cite{Ser:1996:0}, \cite{GodRav:1991:0}).  

In this contribution, a two-velocities lattice Boltzmann scheme is used to approximate the solution of this Cauchy problem.

\subsection{The framework of d'Humi\`eres}
\label{sec:Hum}

In the contribution, we use the notations proposed by d'Humi\`eres in \cite{dHu:1994:0} by considering \(\lattice=\dx\Z\), a regular lattice in one dimension of space with typical mesh size \(\dx\). The time step \(\dt\) is determined after the specification of the velocity scale \(\lambda\) by the relation:  
\begin{equation}\label{eq:lambda}
\dt=\dfrac{\dx}{\lambda}.
\end{equation}
For the scheme denoted by \duqd, we introduce \(\vectv=\lbrace-\lambda,\lambda\rbrace\) the set of the two velocities. The aim of the \duqd scheme is to compute a particle distributions vector
\(\vectf = (\fmnum,\fpnum)\transpose\)
on the lattice \(\lattice\) at discrete values of time: it is a numerical scheme that is formally a discretization in time, space and velocity, where only a finite number of velocities is considered (two, in our contribution), of the Boltzmann equation
\begin{equation*}
\drondt \stylecontinue f(\vart, \varx, c) + c \drondx \stylecontinue f(\vart, \varx, c) = \mathcal{Q}(\stylecontinue f)
\end{equation*}
(even if it cannot be used to simulate this partial differential equation), with a specific collision operator \(\mathcal{Q}(\stylecontinue f)\), whose effect is to relax the particle distributions \(\stylecontinue f\) towards its equilibrium value \(\stylecontinue f^{\equilibre}\).

The scheme splits into two phases for each time iteration: first, the relaxation phase that is local in space (corresponding to the consideration of the collision operator \(\mathcal{Q}\)), and second, the transport phase for which an exact characteristic method is used.

In the framework proposed by d'Humi\`eres \cite{dHu:1994:0}, the relaxation phase reads as a linear relaxation towards the equilibrium that is diagonal into a peculiar base. The vectors of this base are called ``moments'', the terminology being taken from the kinetic theory. Moreover, the equilibrium is {\itshape a priori\/} a nonlinear function of the conservative variables.

In the following, we denote by \(\vectm = (\incun,\incvn)\transpose\) the moments defined for each space point \(\varx\in\lattice\) and for each time \(\vart\) by
\begin{equation}
 \incun = \fmnum + \fpnum, \qquad \incvn = \lambda \left( -\fmnum + \fpnum\right).
\end{equation}
The matrix of the moments \(\matM\) such that \(\vectm = \matM\vectf\) satisfies
\begin{equation}
 \matM = \begin{pmatrix}
 1&1\\ -\lambda&\lambda
 \end{pmatrix},\qquad
 \imatM = \begin{pmatrix}
 1/2 & -1/(2\lambda) \\ 1/2 & 1/(2\lambda)
 \end{pmatrix}.
\end{equation}

Let us now describe one time step of the scheme. The starting point is the particle distributions vector \(\vectf(\varx,\vart)\) in \(\varx\in\lattice\) at time \(\vart\), the moments being computed by
\begin{equation}\label{eq:s1_ftom}
 \vectm(\varx, \vart) = \matM \vectf(\varx,\vart).
\end{equation}
The relaxation phase then reads
\begin{equation}\label{eq:s1_relaxation}
 \incus(\varx,\vart) = \incun(\varx,\vart), 
 \qquad 
 \incvs(\varx,\vart) = \incvn(\varx,\vart) + s (\incve(\varx,\vart) - \incvn(\varx,\vart)),
\end{equation}
where \(s\) is the constant relaxation parameter and \(\incve\) the equilibrium of the second moment, which is considered to be a function of \(\incun\). As a consequence, the first moment \(\incun\) is conserved during the relaxation phase. The relaxation parameter \(s\) is usually taken in \((0,2]\), for stability reasons.

The particle distributions after the relaxation phase are then computed by
\begin{equation}\label{eq:s1_mtof}
 \vectfs (\varx,\vart) = \imatM \vectms(\varx,\vart).
\end{equation}
The transport phase finally reads
\begin{equation}\label{eq:s1_transport}
 \fj(\varx,\vart+\dt) = \fjs(\varx-\epsilon\dx,\vart), \qquad \epsilon\in\{\moins,\plus\}.
\end{equation}

In order to be consistent with Eq.~\eqref{eq:s1_edp}, we impose that \(\incve=\phi(\incun)\) \cite{Gra201400}.

\subsection{A finite volume formalism} 

In order to study the convergence of the scheme, we rewrite it into a finite volume formalism in this section. We first introduce usual finite volume notations. We note \((\xj)_{j\in\Z}\) the sequence of the discrete points in space that make up the lattice \(\lattice\) and \((\tn)_{n\in\N}\) the sequence of the discrete times, with
\begin{equation*}
 \xj = j \dx, \quad j\in\Z, \qquad \tn = n \dt, \quad n\in\N.
\end{equation*}
The scheme can be described either in terms of the particle distributions variables or in terms of the moments. It computes \((\fmjn, \fpjn)_{j\in\Z, n\in\N}\) and \((\ujn, \vjn)_{j\in\Z, n\in\N}\), which are respectively approximations of the local averages of the particle distributions and of the moments in each volume \([\xj,\xj[j+1]]{\times}[\tn,\tn[n+1]]\).

According to \eqref{eq:s1_ftom}, the particle distributions and the moments are linked by the relations
\begin{equation*}
\left\lbrace
\begin{aligned}
 \ujn &= \fmjn+\fpjn, \\
 \vjn &= -\lambda \fmjn + \lambda\fpjn,
\end{aligned}\right.
\qquad
\left\lbrace
\begin{aligned}
 \fmjn &= \tfrac{1}{2} \ujn - \tfrac{1}{2\lambda}\vjn, \\
 \fpjn &= \tfrac{1}{2} \ujn + \tfrac{1}{2\lambda}\vjn, 
 \end{aligned}\right.
\qquad j\in\Z, \ n\in\N.
\end{equation*}

We denote by \((\fmjnpd, \fpjnpd)\) (\textsl{resp.} \((\ujnpd, \vjnpd)\)) the approximated particle distributions functions (\textsl{resp.} moments) at time \(\tn\) after the relaxation step.
By using these notations, according to \eqref{eq:s1_relaxation} the relaxation step of the scheme reads
\begin{equation}\label{eq:sc_uv_relstep}
\ujnpd = \ujn, \qquad
\vjnpd = (1-s)\vjn+s\phi(\ujn),
\qquad j\in\Z, \ n\in\N.
\end{equation}
According to \eqref{eq:s1_mtof}, the particle distributions after the relaxation step are then defined by
\begin{equation}\label{eq:sc_f_relstep}  
 \fmjnpd = \tfrac{1}{2} \ujn - \tfrac{1}{2\lambda}\vjnpd, \qquad
 \fpjnpd = \tfrac{1}{2} \ujn + \tfrac{1}{2\lambda}\vjnpd, 
 \qquad j\in\Z, \ n\in\N,
\end{equation}  
and the transport phase \eqref{eq:s1_transport} now reads
\begin{equation}\label{eqq:sc_f_trstep}
\fpjnp = \fpjnpd[j-1], \qquad  \fmjnp= \fmjnpd[j+1],
\qquad j\in\Z, \ n\in\N.
\end{equation} 

We can then describe a complete step of the scheme in terms of the approximated distribution functions \((\fmjn, \fpjn)\) by
\begin{equation*}
\left\lbrace
\begin{aligned}
\fmjnp &= (1-\tfrac{s}{2}) \fmjn[j+1] + \tfrac{s}{2} \fpjn[j+1] - \tfrac{s}{2\lambda} \phi(\fmjn[j+1]+\fpjn[j+1]), \\
\fpjnp &=  \tfrac{s}{2} \fmjn[j-1] + (1-\tfrac{s}{2}) \fpjn[j-1] + \tfrac{s}{2\lambda} \phi(\fmjn[j-1]+\fpjn[j-1]), 
& j\in\Z,  \ n\in\N,
\end{aligned}
\right. 
\end{equation*}
and in terms of the approximated moments \((\ujn, \vjn)\) by
\begin{equation}\label{eq:sc_one_step_uv}
  \left\lbrace
  \begin{multlined}
    \ujnp = \tfrac{1}{2} (\ujn[j+1]+\ujn[j-1]) - \tfrac{1}{2\lambda}(\vjn[j+1]-\vjn[j-1])\qquad\qquad\qquad\qquad\\
    \shoveright{+\tfrac{s}{2\lambda}\bigl((
        \phi(\ujn[j-1])-\vjn[j-1])
       - (
        \phi(\ujn[j+1])-\vjn[j+1])
      \bigr),}\\
   \shoveleft{\vjnp = \tfrac{1}{2} (\vjn[j+1]+\vjn[j-1]) - \tfrac{\lambda}{2}(\ujn[j+1]-\ujn[j-1])}\\
   +\tfrac{s}{2}\bigl(
        (\phi(\ujn[j-1])-\vjn[j-1])
      + (\phi(\ujn[j+1])-\vjn[j+1])
      \bigr),
  \end{multlined} \right.
  \end{equation}
  for $j\in\Z,\ n\in\N$, the discrete initial data \((\ujz, \vjz)_{j\in\Z}\) being taken as an approximation of \((\incu^0,\phi(\incu^0))\) that will be specified below.

Since $\lambda=\tfrac{\dx}{\dt}$, equation \eqref{eq:sc_one_step_uv} can be rewritten as  
\begin{equation}\label{eq:sc_one_step_con}
\left\lbrace
\begin{multlined}
 \frac{\ujnp-\ujn[j]}{\dt} 
  + \frac{\vjn[j+1]-\vjn[j-1]}{2\dx} 
  - \lambda\frac{\ujn[j+1]-2\ujn[j]+\ujn[j-1]}{2\dx}
  \qquad\qquad\qquad\qquad\\
\shoveright{
=\frac{s}{2\lambda\dt}
  \Bigl[ 
    \bigl(
      \phi(\ujn[j-1])-\vjn[j-1]
    \bigr) - \bigl(
      \phi(\ujn[j+1])-\vjn[j+1]
    \bigr)
    \Bigr],}\\
\shoveleft{\frac{\vjnp-\vjn[j]}{\dt} 
  + \lambda^2\frac{\ujn[j+1]-\ujn[j-1]}{2\dx} 
  - \lambda\frac{\vjn[j+1]-2\vjn[j]+\vjn[j-1]}{2\dx}}\\
=\frac{s}{2\dt}
  \Bigl[ 
    \bigl(
      \phi(\ujn[j-1])-\vjn[j-1]
    \bigr) + \bigl(
      \phi(\ujn[j+1])-\vjn[j+1]
    \bigr)
    \Bigr].
  \end{multlined}
\right.
\end{equation}

Let us now make a link between the \duqd scheme and a relaxation approximation of the scalar conservation law. For given \(\Lambda\) and \(\jinxinparam>0\), let us consider the Jin and Xin relaxation system \cite{JinXin:1995:0}:
\begin{equation}
\label{eq:sys_JX_nh}
  \left\lbrace
  \begin{aligned}
&\drondt \incu +  \drondx \incv = 0,\\
&\drondt \incv + \Lambda^2 \drondx \incu = \frac 1\jinxinparam(\phi(\incu)-\incv).
  \end{aligned}
  \right.
\end{equation}
The \duqd scheme can be linked to a discretization of the relaxation system \eqref{eq:sys_JX_nh}, by a splitting between the hyperbolic part---discretized with a Lax-Friedrichs scheme---and the relaxation part---discretized with an explicit Euler method---where first \(\Lambda=\lambda\) and second the relaxation parameter \(s\), the time step \(\dt\), and  the Jin and Xin relaxation rate \(\jinxinparam\) are linked by the relation \(\jinxinparam=\dt/s\). 

Relaxation schemes corresponding to a discretization of \eqref{eq:sys_JX_nh} for fixed relaxation rate $\jinxinparam$ were widely studied in the past (\cite{Nat:1996:0}, \cite{AreNat:1996:0} and \cite{LatSer:2001:0}). In the first two papers the authors study the convergence of a discretization of \eqref{eq:sys_JX_nh}, as the mesh step goes to 0, then as the relaxation rate $\jinxinparam$ goes to 0. In \cite{LatSer:2001:0}, the convergence of a relaxation scheme is studied as both parameters, the mesh step and the relaxation rate $\jinxinparam$, go to 0 independently.

In the \duqd scheme, both parameters (mesh step and relaxation rate \(\jinxinparam\)) are linked via the new parameter $s$. Regarding the convergence of the \duqd scheme can thus also be viewed as studying the limit of the scheme as both the space mesh and the relaxation rate \(\jinxinparam=\dt/s\) go to 0 at the same time as the relaxation parameter $s$ is constant.

We emphasise  however that the \duqd scheme is not constructed as a discretization of the relaxation system \eqref{eq:sys_JX_nh}. In comparison with the previously mentioned works, in addition to the differences concerning the role of the parameters, it introduces a different (and {\it not natural} in a finite volume framework) discretization of the source term, the discretization of \eqref{eq:sys_JX_nh} considered in \cite{LatSer:2001:0} being :
\begin{equation*}
\left\lbrace
\begin{aligned}
 & \frac{\ujnp-\ujn[j]}{\dt} 
  + \frac{\vjn[j+1]-\vjn[j-1]}{2\dx} 
  - \Lambda\frac{\ujn[j+1]-2\ujn[j]+\ujn[j-1]}{2\dx}=0\\
&\frac{\vjnp-\vjn[j]}{\dt} 
  + \Lambda^2\frac{\ujn[j+1]-\ujn[j-1]}{2\dx} 
  - \Lambda\frac{\vjn[j+1]-2\vjn[j]+\vjn[j-1]}{2\dx}=\frac{1}{\jinxinparam}
    \bigl(
      \phi(\ujnp)-\vjnp\bigr).
  \end{aligned}
\right.
\end{equation*}

\subsection{Comments on the choice of the formalism} 

As the lattice Boltzmann method was introduced as a pure algorithm and not as a specific discretization of a partial differential equation, the interpretation of the discret quantities (here $\ujn$ and $\vjn$ for instance) as node or cell informations is left up to the user.

Note that recent works manage to rewrite any lattice Boltzmann schemes as a multi-step algorithm only on the conserved variables, that is, for instance in our case, without using the non-conserved variable $\vjn$ \cite{bellotti_finite_2022}. 
It yields for the \duqd scheme
$$
  \ujnp = \tfrac{2-s}{2} (\ujn[j+1]+\ujn[j-1]) - (1-s)\ujnm
  - \frac{s}{2\lambda} (\phi(\ujn[j+1]) - \phi(\ujn[j-1])).
$$
In our present contribution, we do not use this property because it is not really adapted to this work where the dynamics on the particule distribution functions has to be investigated.
However, for smooth solutions, a multi-step finite difference framework can then be used to prove that the algorithm is consistent with the target hyperbolic equation and moreover that the equivalent equation of the scheme---in the sense of the finite difference schemes---is the same that the one obtained by the Taylor expansion method \cite{Dub:2008:0, Gra201400}. This equivalent equation reads
\begin{equation}\label{eq:eqeq}
\drondt\incun + \drondx \phi(\incun) = \dt
\bigl(\tfrac{1}{s}-\tfrac{1}{2}\bigr) \drondx \bigl(
  (\lambda^2 - (\phi'(\incun))^2) \drondx\incun
\bigr) + \mathcal{O}(\dt^2).
\end{equation}
It points out that, if the relaxation parameter $s$ lies in $(0, 2)$, the right member acts as a smoothing second-order operator for large enough scheme velocity $\lambda$.

However, as we are interested in discontinuous solutions that can be involved in the non-linear hyperbolic equations, we retain a cell-centered finite volume interpretation of the discret quantities.
The finite volume formalism of the \duqd scheme allowed us to use specific techniques of finite volume schemes for hyperbolic conservation laws, whose solutions might be discontinuous, to prove the convergence of the discrete solution to the unique entropy solution of the scalar conservation law. At our knowledge, we can not use the same techniques for a general lattice Boltzmann scheme and we do not claim that the lattice Boltzmann schemes are finite volume schemes, in the sense that they are not built as a finite volume discretisation of a particular partial differential equation. We just interpret the discrete quantities in the finite volume formalism in order to use the same objects and techniques of these methods to prove our convergence result. 

\section{Total variation and \texorpdfstring{\(\Linf\)}{supremum} estimations}
\label{sec:TVLinf_est}

In this section we establish \(\Linf\) and \(\BV\) estimates for the numerical scheme~\eqref{eq:sc_one_step_uv}. We begin this section by recalling the definitions of the functional spaces that we will use in the sequel and by establishing the assumptions we make in order to get these estimates.  

\subsection{Notations and definitions}

Let us introduce some notations useful in the sequel. 
If \(\incwn=(\incwn_j)_{j\in\Z}\in\R^\Z\), we denote by 
\(\deltax\incwn:=(\tvx{\incwn}{j+\onehalf})_{j\in\Z}\) 
the sequence of \(\R^\Z\) defined by 
\begin{equation*}
\tvx{\incwn}{j+\onehalf} 
  = \incwn_{j+1} - \incwn_j,
\qquad j\in\Z. 
\end{equation*}
If \(\incwn=(\incwn^n)_{n\in\N}\in\R^\N\), we denote by 
\(\deltat\incwn:=(\tvt{\incwn}{n+\onehalf})_{n\in\N}\) 
the sequence of \(\R^\N\) defined by
\begin{equation*}
\tvt{\incwn}{n+\onehalf}
  =\incwn^{n+1}-\incwn^n,
\qquad n\in\N.
\end{equation*}

In order to investigate the convergence of the numerical solution, 
we recall some classical normed sub-spaces of \(\R^\Z\) and the corresponding norms.

\begin{definition}[\(\ell^1(\R^\Z)\) space, \(\normeun{\cdot}\) norm]
Let \(\incwn = (\incwn_j)_{j\in\Z}\in \R^\Z\) and \(\dx>0\) being given. We define the {\em sequential one norm} by
 \begin{equation*}
 \normeun{\incwn} = \dx \sum_{j\in\Z} \vert \incwn_j\vert ,
 \end{equation*}
and the associated space \(\ell^1(\R^\Z)\) as
\begin{equation*}
 \ell^1(\R^Z) = \bigl\lbrace 
 \incwn \in\R^\Z : \ \normeun{\incwn} < +\infty
 \bigr\rbrace.
\end{equation*}
\end{definition}

\begin{definition}[sequential total variation]
Let \(\incwn = (\incwn_j)_{j\in\Z}\in\R^\Z\).
We define the {\em total variation in space} by 
 \begin{equation*}
 \tv(\incwn) = \sum_{j\in\Z} \vert  \tvx{\incwn}{j+\onehalf} \vert .
 \end{equation*}
\end{definition}

We recall below the classical Banach space of functions with bounded variation and the compactness Helly's theorem that we will use in the proof of our main result. To more details and to the proofs of these results, we address \textsl{e.g.} to \cite{Giu:1984} and to \cite{AmbFusPal:2000}.

\begin{definition}[Bounded variation space]
Let \(\Omega\) be an open set of \(\R^n\).
For \(g\in\Lunloc(\Omega)\), the {\em total variation} of \(g\) over \(\Omega\) denoted by \(\tv_\Omega(g)\) is defined by
\begin{equation*}
\tv_\Omega(g) = 
\sup_{\substack{\varphi\in\Coun(\Omega) \\ \norme{\varphi}_\infty\leq 1}}
\int_\Omega g \divergence \varphi.
\end{equation*}
We denote by \(\BV(\Omega)\) the subspace of \(\Lunloc(\Omega)\) of functions with bounded total variation over \(\Omega\):
\begin{equation*}
\BV(\Omega) = \bigl\lbrace
g\in\Lunloc(\Omega): \quad \tv_\Omega(g) < +\infty
\bigr\rbrace.
\end{equation*}
\end{definition}

\begin{proposition*}[Functional space]
The space \(\Lun(\Omega)\cap\BV(\Omega)\) is a Banach space 
for the norm
\begin{equation*}
\norme{g} = \norme{g}_{\Lun(\Omega)} + \tv_\Omega(g).
\end{equation*}
\end{proposition*}

\begin{theorem*}[Helly's sequential theorem]
Let \(\Omega\subset\R^n\) be a bounded open set with a Lipschitz boundary and let \({(g_n)}_{n\in\N}\) be a sequence which is bounded in \(\Lun(\Omega)\cap\BV(\Omega)\). Then there exists a sub-sequence denoted by 
\((g_{\varphi(n)})_{n\in\N}\) and a function \(g\in\Lun(\Omega)\cap\BV(\Omega)\) such that
\begin{equation*}
\left\lbrace
\begin{aligned}
&g_{\varphi(n)} \longrightarrow g \ \text{in }\Lun(\Omega) \text{ and } \textup{a.e. } \varx\in\Omega,\\
&\tv_\Omega(g)\leq \liminf_{n\in\N} \;\tv_\Omega (g_{\varphi(n)}).
\end{aligned}
\right.
\end{equation*}
\end{theorem*}

In the sequel, we use the following notations for the sequences in \(\R^\Z\)
\begin{equation*}
  \fpmn = \bigl( \fpmjn \bigr)_{j\in\Z},
\quad  \un  = \bigl( \ujn  \bigr)_{j\in\Z},
\quad \vn  = \bigl( \vjn  \bigr)_{j\in\Z},
\quad \phi(\un)  = \bigl( \phi(\ujn)  \bigr)_{j\in\Z},
\end{equation*}
for \(n\in\N\), \(\pm\in\lbrace +,-\rbrace\), and we use the following notations for the sequences in \(\R^\N\)
\begin{equation*}
 \fpmj = \bigl( \fpmjn \bigr)^{n\in\N},
\quad \uj  = \bigl( \ujn  \bigr)^{n\in\N},
\quad \vj  = \bigl( \vjn  \bigr)^{n\in\N},
\quad \phi(\uj)  = \bigl( \phi(\ujn)  \bigr)^{n\in\N},
\end{equation*}
for \(j\in\Z\), \(\pm\in\lbrace +,-\rbrace\).

Let us also introduce the functions 
\begin{equation}\label{defhphm} 
\hp(\xi) = \frac{\lambda\xi+\phi(\xi)}{2\lambda},
\qquad
\hm(\xi) = \frac{\lambda\xi-\phi(\xi)}{2\lambda},
\qquad
\xi\in\R.
\end{equation}
The functions \(\hp\) and \(\hm\) correspond to the equilibrium of the particle distribution functions \(\fp\) and \(\fm\).
From relations \eqref{eq:sc_uv_relstep} and \eqref{eq:sc_f_relstep}, the relaxation phase can be rewritten as
\begin{equation}\label{eq:rel_f_comb_conv}
\fpmjnpd =  (1-s) \fpmjn + s\hpm(\ujn), \quad  
j\in\Z,  \ n\in\N.
\end{equation}

\subsection{Assumptions}

We suppose throughout the paper that the initial condition \(\incu^0\) belongs to \(\Linf(\R)\cap\BV(\R)\). We then define 
\begin{gather*}
 \alpha = \essinf\incu^0,
 \qquad 
 \beta = \esssup \incu^0,\\
 M = \max\bigl\lbrace\vert \phi^\prime(\xi)\vert , \text{ for } \alpha\leq\xi\leq \beta\bigr\rbrace. 
 \end{gather*} 
 and we make the following main assumptions concerning the numerical scheme~\eqref{eq:sc_one_step_uv}.

\begin{hyp}\label{hyp:s}
The relaxation parameter \(s\) lies in \((0,1]\).
\end{hyp}

Often in the applications, the value of the relaxation parameter \(s\) is chosen larger than 1 and often close to 2. 
In the case of the linear \duqd, the scheme remains numerically stable in a \(\ell^2\)-sense for \(s\in[0,2]\) and the first-order numerical diffusion term is proportional to \(1/s-1/2\): the choice \(s=2\) minimises then the error and the convergence rate is equal to 2 in that case (see \cite{Gra201400}). 
For general lattice Boltzmann schemes, the ``optimal'' choices for the relaxation parameters are more complicated and are motivated by a combination of stability and accuracy reasons. However, these optimal parameters are generally larger than 1.
The reason why we impose that \(s\) lies in \((0,1]\) is that the scheme has monotonicity properties in this case. And these properties are essential for dealing with weak-solutions according to our technique of estimates. 

\begin{hyp}\label{hyp:CFL}
The velocity of the scheme satisfies \(\lambda=\dx/\dt \geq M\).
\end{hyp}

This assumption is known as the {\itshape sub-characteristic condition\/}. It states that the solutions of the equilibrium equation propagate with a characteristic speed smaller than that of the numerical scheme.

\begin{hyp}\label{hyp:initeq}
The initial state is given by
\begin{equation*}
 \ujz = \frac{1}{\dx}
   \int_{\xj}^{\xj[j+1]}\incu^0(\varx) \ddd\varx,
 \qquad \vjz = \phi(\ujz),
 \qquad j\in\Z.
\end{equation*}
\end{hyp}

This last assumption means that the initial values of the non-conserved moment are taken to the equilibrium. This choice is done more often than not as this second moment has to be a perturbation of the equilibrium state \cite{LalLuo:2000:0}.

\subsection{Preliminary lemmas} 

We begin by proving two lemmas that will be used throughout the proofs of the estimates and convergence of the numerical solutions. The first one concerns the total variation of the numerical initial data.

\begin{lemma}
  \label{lem_VTu0}
 The total variation of the discretized conserved moment at initial time \(\un[0]=(\ujz)_{j\in\Z}\) is controlled by the total variation of the initial function \(\incu^0\):
 \begin{equation*}
 \tv(\un[0]) = \sum_{j\in\Z}\vert \tvx{\incun^0}{j+\f}\vert  \leq \tv_\R(\incu^0).
 \end{equation*}
\end{lemma}

\begin{proof}
This result is commonly used in the context of finite volumes approximation of conservation laws; we propose a precise proof in appendix \ref{ap_lem_VTu0}.
\end{proof}

The second lemma concerns some properties of the functions \(\hp\) and \(\hm\) defined by \eqref{defhphm}.

\begin{lemma}\label{thm:hphm}
  Under assumption \ref{hyp:CFL}, the functions \(\hp\) and \(\hm\) are non decreasing over \([\alpha, \beta]\). Moreover,
    for all $\xi\in\R,\ $\(\hp(\xi)+\hm(\xi)=\xi\) and  
      \(\hpprime(\xi)+\hmprime(\xi)=1\).
\end{lemma}

\begin{proof}
 The functions \(\hp\) and \(\hm\) are \(\Cun(\R)\) and their derivative is non negative over \([\alpha,\beta]\), under assumption~\ref{hyp:CFL}. The second result of the lemma is a trivial consequence of the definition of the functions \(\hpm\).
\end{proof}

\subsection{Uniform bound estimates}

In this section we establish a maximum principle for the numerical solutions of the scheme \eqref{eq:sc_one_step_uv}. The technique of the proof is the same as in \cite{AreNat:1996:0} or \cite{LatSer:2001:0}. For the clearness of our work we present it here.

\begin{proposition}[maximum principle]\label{th:proplinfty}
Under assumptions~\ref{hyp:s}, \ref{hyp:CFL} and \ref{hyp:initeq}, we have
\begin{equation*}
 \ujn\in[\alpha,\beta], 
\qquad
 \fpmjn\in[\hpm(\alpha),\hpm(\beta)] ,
\qquad j\in\Z, \ n\in\N.
\end{equation*}
\end{proposition}

\begin{proof}
A recursive reasoning is done. Since \(\incu^0\in[\alpha,\beta], a.e.\ \varx\in\R\), we have that 
\(\ujz\in[\alpha,\beta]\) for \(j\in\Z\). 
By lemma~\ref{thm:hphm}, and since \(\vjz=\phi(\ujz)\), we get
\begin{equation*}
\fpmjz = \hpm(\ujz)\in[\hpm(\alpha),\hpm(\beta)],
\qquad j\in\Z.
\end{equation*}

We then assume that the three inclusions are true for a certain \(n\in\N\).
As \(s\in(0,1]\), Eq.~\eqref{eq:rel_f_comb_conv} implies that \(\fmjnpd\) and \(\fpjnpd\) are respectively convex linear combinations of \(\fmjn\) and \(\hm(\ujn)\), and of \(\fpjn\) and \(\hp(\ujn)\), so that 
\(\fmjnpd\in[\hm(\alpha),\hm(\beta)]\)
and
\(\fpjnpd\in[\hp(\alpha),\hp(\beta)]\), for \(j\in\Z\).
The transport phase just shift the distribution functions, so that the same inclusions yield for \(\fmjnp\) and \(\fpjnp\) for \(j\in\Z\). Finally, we have
\begin{equation*}
\ujnp = \fmjnp + \fpjnp \in [ \hm(\alpha)+\hp(\alpha) , \hm(\beta)+\hp(\beta) ] = [\alpha, \beta],
\qquad j\in\Z,
\end{equation*}
as \(\hm(\xi)+\hp(\xi)=\xi\) for all \(\xi\in\R\), by lemma \ref{thm:hphm}.
\end{proof}

\subsection{Total variation estimates}

We establish now estimates on the total variation in space and in time of the numerical solutions. This kind of estimates use now classical tools in the framework of finite volume schemes for hyperbolic conservation laws (see for instance \cite{CraMaj:1980:0}), and are similar but slightly different to the ones of \cite{AreNat:1996:0} and of \cite{Cae:2006:0}. For the completeness of our work we present the proofs here.

\begin{proposition}[Spatial total variation estimations]\label{th:propVT}
Under assumptions~\ref{hyp:s}, \ref{hyp:CFL} and \ref{hyp:initeq},
the particle distributions functions satisfy the total variation decreasing estimate
\begin{equation}\label{eq:tvd_f}
 \tv(\fpnp) + \tv(\fmnp) \leq  \tv(\fpn) + \tv(\fmn), \qquad n\in\N.
\end{equation}
Moreover, we have
\begin{align}
\label{eq:tv_f}
&\tv(\fpn)+\tv(\fmn)\leq \tv_\R(\incu^0), &n\in\N,\\
\label{eq:tv_u}
&\tv(\un)\leq \tv_\R(\incu^0), & n\in\N,\\
\label{eq:tv_v}
&\tv(\vn)\leq \lambda\tv_\R(\incu^0), & n\in\N.
\end{align}
\end{proposition}

\begin{proof}
First we remark that we have
\(\tv(\un) \leq \tv(\fmn)+\tv(\fpn)\), as \(\ujn=\fmjn+\fpjn\). 
We first evaluate the total variation of the approximated particle  distributions functions after the relaxation phase. 
We deduce from \eqref{eq:rel_f_comb_conv} that
\begin{equation*}
  \tvx{\fpmn[n+\f]}{j+\f} 
  = (1-s) \tvx{\fpmn}{j+\f} 
  + s \bigl(
    \hpm(\ujn[j+1]) - \hpm(\ujn)
  \bigr).
\end{equation*}
And we have, by performing a first order Taylor expansion, 
\begin{multline*}
\hpm(\ujn[j+1]) - \hpm(\ujn) 
= \frac{
  \lambda\tvx{\un[n]}{j+\f}
  \pm(\phi(\ujn[j+1]) 
  - \phi(\ujn))
}{2\lambda} \\
= \frac{
  (\lambda \pm \phi^\prime(\xi_{j+\f}^n)) \tvx{\un[n]}{j+\f}
}{2\lambda}
= \hpmprime(\xi_{j+\f}^n) \tvx{\un[n]}{j+\f},
\end{multline*}
where \(\xi_{j+\f}^{n}\in(\alpha,\beta)\) is such that
\(\phi(\ujn[j+1])-\phi(\ujn) = \phi^\prime(\xi_{j+\f}^n)(\ujn[j+1]-\ujn)\). We then have, using that \(\hpmprime\ge0\) over \([\alpha,\beta]\) and that \(s\in(0,1]\), 
\begin{equation*}
 \bigl\vert  \tvx{\fpmn[n+\f]}{j+\f} \bigr\vert 
 \leq (1-s) \bigl\vert  \tvx{\fpmn}{j+\f} \bigr\vert  
 + s \hpmprime(\xi_{j+\f}^{n}) \bigl\vert  \tvx{\un}{j+\f}\bigr\vert .
\end{equation*}
Since \(\hpprime+\hmprime=1\), summing the above two inequalities over \(j\in\Z\), for \(\pm\in\{+,-\}\), yields
\begin{equation*}
 \tv(\fpnpd) + \tv(\fmnpd) \leq (1-s) \tv(\fpn) + (1-s) \tv(\fmn)
 + s \tv(\un),
\end{equation*}
and finally, 
\begin{equation*}
 \tv(\fpnpd) + \tv(\fmnpd) \leq  \tv(\fpn) + \tv(\fmn).
\end{equation*}
Concerning the transport phase, the total variation of the particle distributions functions after this phase is unchanged as they are defined as a translation of the particle distributions before the phase. We have thus
\begin{equation*}
\tv(\fpnp) + \tv(\fmnp) = \tv(\fpnpd) + \tv(\fmnpd) \leq  \tv(\fpn) + \tv(\fmn),
\end{equation*}
which allow us to conclude that
\begin{equation}\label{eq_vtfpm_vtf0}
\tv(\fpn) + \tv(\fmn)\leq \tv(\fmn[0]) + \tv(\fpn[0]),\qquad \forall n \in \N.
\end{equation}

We also have for the total variation of the conserved moment \(\un\)
\begin{equation}\label{eq_vtu_vtu0}
 \tv(\un)\leq \tv(\fmn)+\tv(\fpn) \leq \tv(\fmn[0]) + \tv(\fpn[0]).
\end{equation}
The definition of the initial numerical data in assumption~\ref{hyp:initeq} implies that
\begin{equation*}
 \bigl\vert \tvx{\fpmn[0]}{j+\f}\bigr\vert  
 = \bigl\vert \hpm(\ujz[j+1]) - \hpm(\ujz)\bigr\vert  
 = \hpmprime(\xi_{j+\f}^0) \bigl\vert \tvx{\un[0]}{j+\f}\bigr\vert .
\end{equation*}
By summing these two equalities over \(j\in\Z\), for \(\pm\in\{+,-\}\), we get 
\(\tv(\fmn[0]) + \tv(\fpn[0]) = \tv(\un[0])\), which combined with \eqref{eq_vtfpm_vtf0} and with \eqref{eq_vtu_vtu0} gives \eqref{eq:tv_f} and \eqref{eq:tv_u}.
Since \(\vert \vn\vert \le\lambda(\vert \fpn\vert +\vert \fmn\vert )\), we also obtain \eqref{eq:tv_v}.
\end{proof}

In order to control the total variation of the numerical approximation in time and space variables, we prove now uniform total time variation estimates for the approximated solutions. To do so, we estimate the quantity
\begin{equation*}
\sum_{j\in\Z} 
 \left\vert \tvt{\fpj}{n+\f}\right\vert 
+
\left\vert \tvt{\fmj}{n+\f}\right\vert .
\end{equation*}

\begin{proposition}[total variation in time estimates]
\label{th:propVTt}
Under assumptions \ref{hyp:s}, \ref{hyp:CFL} and \ref{hyp:initeq},
the particle distributions functions satisfy the estimate
\begin{multline}\label{eq:tvt_f}
  \sum_{j\in\Z}
  \left\vert \tvt{\fpj}{n+\f}\right\vert 
  +
  \left\vert \tvt{\fmj}{n+\f}\right\vert \\
  \leq 
  \sum_{j\in\Z}
  \left\vert \tvt{\fpj}{n-\f}\right\vert 
  +
  \left\vert \tvt{\fmj}{n-\f}\right\vert\leq 2\tv_\R(\incu^0), \qquad n\in\N.
\end{multline}
Moreover we have 
\begin{equation}\label{eq:tvt_u}
\sum_{j\in\Z}\left\vert \tvt{\uj}{n+\f}\right\vert \leq 2\tv_\R(\incu^0), 
\qquad
\sum_{j\in\Z}\left\vert \tvt{\vj}{n+\f}\right\vert 
  \leq 2\lambda\tv_\R(\incu^0), 
\qquad n\in\N.
\end{equation}
\end{proposition}

\begin{proof}
Suppose that
\begin{equation*}
\sum_{j\in\Z}
\left\vert \tvt{\fpj}{n-\f}\right\vert 
+ \left\vert \tvt{\fmj}{n-\f}\right\vert 
< +\infty.
\end{equation*} 
Since \(\incun=\fp+\fm\), then we also have 
\begin{equation*}
\sum_{j\in\Z}\left\vert \tvt{\uj}{n-\f}\right\vert < +\infty.
\end{equation*}
From \eqref{eqq:sc_f_trstep} and \eqref{eq:rel_f_comb_conv}, we have \(\fpmjnp=
\fpmjnpd[j\mp1]=(1-s)\fpmjn[j\mp1]+s\hpm(\ujn[j\mp1])\). We obtain then
\begin{multline*}
\tvt{\fpmj}{n+\f}
  = (1-s) \tvt{\fpmj[j\mp1]}{n-\f}
    + s \bigl(
      \hpm(\ujn[j\mp1]) - \hpm(\ujnm[j\mp1])
    \bigr) \\
  = (1-s) \tvt{\fpmj[j\mp1]}{n-\f}
    + s \frac{\lambda\pm\phi^\prime(\xi_{j\mp1}^{n-\onehalf})}{2\lambda}\tvt{\uj[j\mp1]}{n-\f},
\end{multline*}
where we used that, for \(j\in\Z\),
\begin{equation*}
\hpm(\ujn)-\hpm(\ujnm) = 
\hpmprime(\xi_j^{n-\onehalf})\tvt{\uj}{n-\f} = 
  \frac{\lambda\pm\phi^\prime(\xi_j^{n-\onehalf})}{2\lambda}
  \tvt{\uj}{n-\f},
\end{equation*}
with \(\xi_j^{n-\onehalf}\) lying between \(\ujn\) and \(\ujnm\).

Assumptions  \ref{hyp:s} and \ref{hyp:CFL} yield
\begin{equation*}
\left\vert \tvt{\fpmj}{n+\f} \right\vert 
\leq (1-s)
\left\vert \tvt{\fpmj[j\mp1]}{n-\f} \right\vert 
+ s \frac{\lambda\pm\phi^\prime(\xi_{j\mp1}^{n-\onehalf})}{2\lambda}
\left\vert \tvt{\uj[j\mp1]}{n-\f} \right\vert ,
\end{equation*}
By summing the above two inequalities over \(j\in\Z\), for \(\pm\in\lbrace+,-\rbrace\), we get 
\begin{multline*}
  \sum_{j\in\Z} \left( 
    \left\vert \tvt{\fpj}{n+\f} \right\vert 
    + \left\vert \tvt{\fmj}{n+\f} \right\vert 
  \right) \\ 
  \leq (1-s) \sum_{j\in\Z} \left(
    \left\vert \tvt{\fpj}{n-\f} \right\vert 
    + \left\vert \tvt{\fmj}{n-\f} \right\vert 
  \right)
  + s \sum_{j\in\Z} \left\vert \tvt{\uj}{n-\f} \right\vert .
\end{multline*}
Since we have 
\(\vert \incun\vert \leq \vert \fpnum \vert + \vert \fmnum \vert\), we get
\begin{equation*}
\sum_{j\in\Z} \left(
  \left\vert \tvt{\fpj}{n+\f} \right\vert 
  + \left\vert \tvt{\fmj}{n+\f} \right\vert 
\right)
\leq \sum_{j\in\Z} \left(
  \left\vert \tvt{\fpj}{n-\f} \right\vert 
  + \left\vert \tvt{\fmj}{n-\f} \right\vert 
\right).
\end{equation*}
We conclude then, by reasoning recursively, that
\begin{equation}\label{eq:locestVTt}
\sum_{j\in\Z} \left(
  \left\vert \tvt{\fpj}{n+\f} \right\vert 
  + \left\vert \tvt{\fmj}{n+\f} \right\vert 
\right)
\leq 
\sum_{j\in\Z} \left(
  \left\vert \tvt{\fpj}{\f} \right\vert 
  + \left\vert \tvt{\fmj}{\f} \right\vert 
\right),\ \forall n \in\N.
\end{equation}
Now, we have, by using assumption~\ref{hyp:initeq},
\begin{multline*}
\tvt{\fpj}{\f} =
\vert \fpju - \fpjz \vert
= \vert (1-s) \fpjz[j-1] + s \hp(\ujz[j-1])-\fpjz \vert \\
= \vert \hp(\ujz[j-1])-\hp(\ujz) \vert 
\leq \vert \ujz[j-1] - \ujz \vert ,
\end{multline*}
and, in the same way,
\begin{equation*}
\tvt{\fmj}{\f} \leq \vert \ujz[j+1] - \ujz \vert,
\end{equation*}
which, combined with \eqref{eq:locestVTt}, implies \eqref{eq:tvt_f}. Since \(\vert \un\vert \le\vert \fpn\vert +\vert \fmn\vert \) and \(\vert \vn\vert \le\lambda(\vert \fpn\vert +\vert \fmn\vert )\), we also obtain \eqref{eq:tvt_u}.
\end{proof}

\subsection{Estimation of the equilibrium gap}

We aim to prove that, at the discrete level, the equilibrium gap \(\phi(\incun)-\incvn\) tends to \(0\), as the discretization step tends to \(0\). The purpose of the next proposition is then to estimate the quantity \(\phi(\ujn)-\vjn\).

\begin{proposition}[estimation of the discrete equilibrium gap]\label{th:prop_eq_gap}
 Under assumptions \ref{hyp:s}, \ref{hyp:CFL} and \ref{hyp:initeq}, the equilibrium gap is bounded, that is
 \begin{equation}\label{eq:lun_v}
 \normeun{\phi(\un)-\vn} \leq \frac{2\lambda\dx}{s} \tv_\R(\incu^0),
 \qquad n\in\N.
 \end{equation}
\end{proposition}

\begin{proof}
We have to estimate
\begin{equation*}
\normeun{\phi(\un)-\vn} 
= \dx\sum_{j\in\Z} \vert \phi(\ujn)-\vjn\vert .
\end{equation*}
For \(n=0\), we have  
\(\normeun{\phi(\un[0])-\vn[0]} = 0\). 
Then, we use the relations
\[\phi(\ujnp)-\phi(\ujn) = \phi^\prime(\zeta_j^{n+\f})(\ujnp-\ujn)\]
with
\(\zeta_j^{n+\f}\in(\alpha,\beta)\) and
\(\vjnpd=(1-s)\vjn + s \phi(\ujn)\). We have
\begin{multline*}
 \phi(\ujnp)-\vjnp = \phi(\ujnp)-\phi(\ujn)+\phi(\ujn)-\vjnpd+\vjnpd-\vjnp\\
 \shoveleft{\qquad = (1-s) \bigl(\phi(\ujn) - \vjn\bigr) + \phi^\prime(\zeta_j^{n+\f})(\ujnp\!\!-\ujn)
  - (\vjnp \!\!- \vjnpd)}\\
 \shoveleft{\qquad = 
 (1-s) \bigl(\phi(\ujn) - \vjn\bigr)
 + \phi^\prime(\zeta_j^{n+\f})\bigl(\fpjnpd[j-1]+\fmjnpd[j+1] -\fpjnpd-\fmjnpd \bigr)}\\
\shoveright{ - \lambda \bigl( \fpjnpd[j-1] - \fmjnpd[j+1] - \fpjnpd + \fmjnpd \bigr)}\\
 \shoveleft{\qquad = 
 (1-s) \bigl(\phi(\ujn) - \vjn\bigr)
 + \bigl(\phi^\prime(\zeta_j^{n+\f})-\lambda\bigr)
   \bigl(\fpjnpd[j-1]-\fpjnpd \bigr)}\\
 + \bigl(\phi^\prime(\zeta_j^{n+\f})+\lambda\bigr)
   \bigl(\fmjnpd[j+1] -\fmjnpd \bigr).
 \end{multline*}
 By assumptions~\ref{hyp:s} and~\ref{hyp:CFL}, we obtain
 \begin{equation*}
\vert \phi(\ujnp)-\vjnp\vert\leq \vert 1-s\vert\,\vert\phi(\ujn) - \vjn\vert+2\lambda\left(\vert \fpjnpd[j-1]-\fpjnpd\vert + \vert \fmjnpd[j+1] -\fmjnpd\vert \right).
 \end{equation*} 
As a consequence of proposition \ref{th:propVT}, summing the above inequalities over \(j\in\Z\) yields
\begin{equation*}
\normeun{\phi(\un[n+1])-\vn[n+1]}
\leq
\vert  1-s\vert  \normeun{\phi(\un)-\vn}
+ 2\lambda \dx \tv(\un[0])
\end{equation*}
as
\(\tv(\fpnpd) + \tv(\fmnpd)\leq\tv(\un[0])\).
A recursive reasoning then implies that
\begin{multline*}
\normeun{\phi(\un)-\vn}
\leq 2\lambda \dx \tv(\un[0])
\sum_{k=0}^n \vert  1-s\vert ^k \\
\leq \frac{2\lambda\dx}{s} \tv(\un[0])
\leq \frac{2\lambda\dx}{s} \tv_\R(\incu^0),
\qquad n\in\N,
\end{multline*}
as \(s\in(0,1]\) by assumption~\ref{hyp:s} and using lemma~\ref{lem_VTu0}.
\end{proof}

\section{Convergence of the numerical scheme}
\label{sec:conv_scheme}

This section is devoted to proving the convergence of the numerical solution towards a weak solution of the nonlinear Cauchy problem (\refeq{eq:s1_edp}-\refeq{eq:s1_ic}). This result is obtained as a consequence of the spatial and temporal estimates proved in section \ref{sec:TVLinf_est}.
The compactness of the numerical sequences is used in the space of the functions with bounded variation and is obtained as a consequence of  Helly's theorem.

As usual in the context of finite volume schemes, we seek an approximated solution of the form
\begin{align}
\udelta(\vart, \varx) = \sum_{j\in\Z}\sum_{n\in\N}
  \ujn \ \indicatrice_{[\tn,\tn[n+1])} \ \indicatrice_{[\xj,\xj[j+1])},\\
\vdelta(\vart, \varx) = \sum_{j\in\Z}\sum_{n\in\N}
  \vjn \ \indicatrice_{[\tn,\tn[n+1])} \ \indicatrice_{[\xj,\xj[j+1])}.
\end{align}
We begin by proving the convergence of the sequence \((\udelta,\vdelta)\), as the space-meshing \(\dx\) and the time-step \(\dt\) tend to 0, towards a function \((\ulim,\vlim)\), where \(\ulim\) is a weak solution of the Cauchy problem (\refeq{eq:s1_edp}-\refeq{eq:s1_ic}) and \(\vlim=\phi(\ulim)\).
The main tools of the proof are classical in the context of finite volume schemes \cite{GodRav:1991:0}, even if the treatment of the source term is specific to the \duqd scheme. We propose the proof in appendix~\ref{ap_proof_conv}.

\begin{theorem}[convergence towards a weak solution]
\label{th:theo_conv}
Under assumptions \ref{hyp:s}, \ref{hyp:CFL} and \ref{hyp:initeq}, 
there exist sub-sequences of \(\udelta\), \(\vdelta\) (still denoted by \(\udelta\), \(\vdelta\)) and functions \(\ulim\), \(\vlim\) with
\begin{equation*}
\ulim, \vlim
\in
\Linf(\R^+{\times}\R) \cap \BV([0,T]{\times}\R),
\end{equation*}
for all \(T>0\), such that 
\begin{equation*}
(\udelta,\vdelta) \xrightarrow[\dt,\dx\to0]{} (\ulim,\vlim)
\quad
\Lunloc([0,+\infty[\times\R)\times \Lunloc([0,+\infty[\times\R).
\end{equation*}

We have in addition that \(\ulim\) is a weak solution of (\refeq{eq:s1_edp}-\refeq{eq:s1_ic}) in \([0,+\infty[\times\R\), and that \(\vlim=\phi(\ulim)\).
\end{theorem}

\section{Entropies and numerical entropy estimates}
\label{sec:num_ent}

In this section we aim to establish discrete entropy estimates for the numerical scheme \eqref{eq:sc_one_step_uv}. To do so, we will use the relaxation entropies introduced in \cite{Ser:2000:0} in order to construct numerical entropies for the scheme. We will also make a link between these relaxation entropies and a kinetic decomposition of the dual entropy for the nonlinear conservation law, introduced in \cite{Bou:2003:0,Dub:2013:0}. 

\subsection{Entropy}

Let us consider an entropy-entropy flux pair \((\ent,\entflux)\), with \(\ent\) strictly convex, for the hyperbolic scalar conservation law \eqref{eq:s1_edp}, 
\begin{equation*}
\ent'(\incu) \phi'(\incu)=\entflux'(\incu),\qquad \ent''(\incu)>0, \quad \incu\in[\alpha,\beta]. 
\end{equation*}
For given \(\lambda\), we introduce the homogeneous Jin and Xin relaxation system \cite{JinXin:1995:0}:
\begin{equation}
\label{eq:sys_JX}
  \left\lbrace
  \begin{aligned}
&\drondt \thincu +  \drondx \thincv = 0,\\
&\drondt \thincv + \lambda^2 \drondx \thincu = 0,
  \end{aligned}
  \right.
\end{equation}
which we can write in an equivalent way in the characteristic variables \((\thfp, \thfm)\):
\begin{equation}\label{eq:sys_JX_car}
  \left\lbrace
  \begin{aligned}
&\drondt \thfp + \lambda \drondx \thfp = 0,\\
&\drondt \thfm - \lambda \drondx \thfm = 0.
  \end{aligned}
  \right.
\end{equation}
We remark that \eqref{eq:sys_JX} is the homogeneous part of \eqref{eq:sys_JX_nh}, with $\Lambda=\lambda$, where $\lambda$ is the velocity scale in the \duqd scheme; we consider here this choice in order to motivate the construction of numerical entropies for our scheme.

Following Serre \cite{Ser:2000:0}, we define an entropy-entropy flux couple \((\entrelax, \entfluxrelax)\) for system~\eqref{eq:sys_JX} as follows. 
Let us first define the couple \((\kinentp,\kinentm)\) of the kinetic entropies by 
\begin{equation*}
 \kinent(\varcharf) = \frac{
   \lambda\ent\pm\entflux
 }{2\lambda} ((\hpm)^{-1}(\varcharf)),
 \qquad \varcharf\in[\hpm(\alpha),\hpm(\beta)],
\end{equation*}
and the couple \((\tilde{\ent},\tilde{\entflux})\) by
\begin{equation*}
(\tilde{\ent},\tilde{\entflux})(\fm,\fp)
=
\bigl(
  \kinentp(\fp)+\kinentm(\fm),
  \lambda \kinentp(\fp) - \lambda \kinentm(\fm)
\bigr),\qquad \fpm\in[\hpm(\alpha),\hpm(\beta)].
\end{equation*}
We remark then that the couple \((\tilde{\ent},\tilde{\entflux})\) is an entropy-entropy flux pair for system~\eqref{eq:sys_JX_car} that lies in the set \([\hp(\alpha),\hp(\beta)]\times[\hm(\alpha),\hm(\beta)]\).

Let us now define the couple \((\entrelax, \entfluxrelax)\) by
\begin{equation*}
\entrelax(\incu,\incv) =
  \kinentp\Bigl(\frac{\lambda \incu+\incv}{2\lambda}\Bigr)
  + \kinentm\Bigl(\frac{\lambda \incu-\incv}{2\lambda}\Bigr),
\quad
\entfluxrelax(\incu,\incv) =
  \lambda \kinentp\Bigl(\frac{\lambda \incu+\incv}{2\lambda}\Bigr)
  - \lambda \kinentm\Bigl(\frac{\lambda \incu-\incv}{2\lambda}\Bigr),
\end{equation*}
for \((\incu,\incv)\) such that 
\((\lambda \incu+\incv)/(2\lambda) \in[\hp(\alpha),\hp(\beta)]\),
\((\lambda \incu-\incv)/(2\lambda) \in[\hm(\alpha),\hm(\beta)]\).
Then, the couple \((\entrelax, \entfluxrelax)\) is an entropy-entropy flux pair for the system~\eqref{eq:sys_JX}, which satisfies
\begin{equation*}
(\entrelax, \entfluxrelax)(\incu,\phi(\incu))
= (\ent,\entflux)(\incu),\quad \forall \incu\in[\alpha,\beta].
\end{equation*}

Let us now state some properties of the entropies $\kinentp$ and $\kinentm$ that are useful to establish the entropy estimates for the numerical solutions.

Due to the assumption \eqref{hyp:CFL}, we can easily check the following lemma.

\begin{lemma}
\label{th:lm_kinent}
  Under assumption \eqref{hyp:CFL}, the functions \(\kinentp\) and \(\kinentm\) satisfy 
\begin{equation*}
 \kinentprime(\varcharf) = \ent^\prime((\hpm)^{-1}(\varcharf)),
 \quad
 \kinentsecond(\varcharf) > 0,
 \quad \varcharf\in[\hpm(\alpha), \hpm(\beta)].
\end{equation*}
Furthermore we have
\begin{equation*}
\kinentprime[+](\hp(\incu))=\kinentprime[-](\hm(\incu))=\ent^\prime(\incu), \quad \incu\in[\alpha,\beta].
\end{equation*}
\end{lemma}
\begin{proof}
  The result of the lemma follows immediately from the definition of \(\kinent\) and from the relation 
  \(\ent^\prime\phi^\prime=\entflux^\prime\). 
  Since \({\hpm}'(\xi)=\frac{\lambda\pm\phi^\prime(\xi)}{2\lambda}\), for \(\xi\in\R\), we have
  \[
    \kinentprime(\varcharf)=\left(\ent^\prime\frac{\lambda\pm\phi^\prime}{2\lambda}\right)({(\hpm)}^{-1}(\varcharf))\,\frac{1}{({\hpm})^{\prime}({(\hpm)}^{-1}(\varcharf))}=
\ent^\prime((\hpm)^{-1}(\varcharf)).
  \]
\end{proof}

\subsection{Dual entropy}

In this paragraph we make a link between the relaxation entropies \(\kinentp, \kinentm\) defined previously with the kinetic decomposition of the dual entropy introduced by Bouchut in \cite{Bou:2003:0} and by one of the authors in \cite{Dub:2013:0}  for the nonlinear conservation law \eqref{eq:s1_edp}.

Let us introduce the entropy variable \(\entvar\) for the scalar conservation law:
\begin{equation*}
\entvar = \ent'(\incu). 
\end{equation*}
We define the dual entropy 
\(\entvar \longmapsto \ent^\dual (\entvar)\)
according to 
\begin{equation*}
\ent^\dual (\entvar) = \sup_{w} \bigl( \entvar w - \ent(w) \bigr).
\end{equation*}
We then have 
\begin{gather*}
\ent^\dual (\entvar) = \entvar (\ent')^{-1}(\entvar)  - \ent \big( (\ent')^{-1}(\entvar) \big) =  \entvar\incu - \ent(\incu),\\
\frac{\ddd\ent^\dual}{\ddd\entvar}(\entvar) = 
(\ent')^{-1}(\entvar) =  \incu.
\end{gather*}
The dual entropy flux \(\entvar \longmapsto \entflux^\dual (\entvar) \) is then defined according to 
\begin{equation*}
\entflux^\dual (\entvar) = 
\entvar \phi\bigl((\ent')^{-1}(\entvar)\bigr) - \entflux\bigl((\ent')^{-1}(\entvar)\bigr)
=  \entvar \phi(\incu) - \entflux(\incu).
\end{equation*}
We have the following differential
\begin{equation*}
\frac{\ddd \entflux^\dual}{\ddd\entvar} (\entvar) = 
\phi \bigl((\ent')^{-1}(\entvar)\bigr)
=  \phi(\incu).
\end{equation*}

Following \cite{Dub:2013:0}, let us introduce two convex functions 
\(\entvar \longmapsto \psipmdual(\entvar)\) 
that satisfy the kinetic decomposition of the dual entropy 
\begin{equation} 
\label{eq:fd-psis-star}
  \psipmdual[+](\entvar) +  \psipmdual[-](\entvar) 
  = \ent^\dual (\entvar), 
  \qquad
  \lambda \psipmdual[+](\entvar) - \lambda \psipmdual[-](\entvar) 
  = \entflux^\dual (\entvar),
\end{equation} 
and the equilibrium functions 
\(\incu \longmapsto \fpmeq (\incu) \) according to the relation 
\begin{equation*}
\fpmeq (\incu)  =  
\frac{\ddd\psipmdual}{\ddd\entvar}
\bigl(\ent'(\incu)\bigr)
= 
\frac{\ddd\psipmdual}{\ddd\entvar}(\entvar).
\end{equation*}

We can then define a kinetic hyperbolic system
\begin{equation}\label{eq:fd-kinetic-d1q2}
  \left\lbrace
  \begin{aligned}
&\drondt \fp + \lambda \drondx \fp = \frac{1}{\epsilon}
\bigl( \fpmeq[+](\fp+\fm) - \fp \big),\\
&\drondt \fm - \lambda \drondx \fm = \frac{1}{\epsilon}
\bigl( \fpmeq[-](\fp+\fm) - \fm \big).
  \end{aligned}
  \right.
\end{equation}

We introduce  the duals \(\varcharf \longmapsto \psipm(\varcharf) \) of the functions 
\(\psipmdual\) defined in \eqref{eq:fd-psis-star}:
\begin{equation*}
\psipm(\varcharf) = \sup_{\entvar} \bigl( \varcharf \entvar - \psipmdual(\entvar) \bigr). 
\end{equation*}
Then 
\begin{equation*}
\psipmprime(\varcharf) = 
\left( \frac{\ddd \psipmdual }{\ddd \entvar} \right)^{-1} (\varcharf) . 
\end{equation*}
We know that such a framework is able to put in evidence a ``H-theorem''  \cite{Dub:2013:0}.
Just multiply each equation of \eqref{eq:fd-kinetic-d1q2} by \(\psipmprime(\fpm)\); then 
\begin{equation*}
 \drondt \big( \psipm[+](\fp) + \psipm[-](\fm) \big) 
+ \drondx \big( \lambda \psipm[+](\fp) - \lambda \psipm[-](\fm) \big) 
 \leq  0. 
\end{equation*}

The natural question is to make a link between this framework and the tools introduced in this contribution, 
\textit{id est} to make a link between \(\hpm\) and \(\fpmeq\) and between \(\kinent\) and \(\psipm\). 
We have the following proposition.

\begin{proposition}[Duality and entropy decomposition]
\label{th:prop-dualite-entropie}
We have 
\begin{align*} 
\fpmeq(\incu) &=  \hpm(\incu), & \incu&\in[\alpha,\beta],\\
\psipm(\varcharf) &=  \kinent(\varcharf), 
&\varcharf&\in[\hpm(\alpha), \hpm(\beta)]. 
\end{align*}    
\end{proposition}

\begin{proof}
We first have
\begin{equation*}
 \psipmdual(\entvar) = \frac{
   \lambda\ent^\dual(\entvar)\pm\entflux^\dual(\entvar)
 }{2\lambda},
 \qquad
 \frac{\ddd\psipmdual}{\ddd\entvar}(\entvar) = \frac{
   \lambda\incu\pm\phi(\incu)
 }{2\lambda} = \hpm(\incu)
 \quad \text{with }\entvar=\ent'(\incu).
\end{equation*}
Then, we can identify \(\fpmeq=\hpm\).

At the optimum value \(\entvar\) that define \(\psipm(\varcharf)\), we have 
\(\varcharf = \big( \psipmdual \big)' (\entvar)  = \hpm (\incu) \) with \(\ent'(\incu) = \entvar\). 
Then 
\begin{align*}
 \psipm(\varcharf) 
 & = \varcharf\entvar - \frac12 \Bigl(
   \ent^\dual(\entvar)\pm\frac{1}{\lambda} \entflux^\dual(\entvar)
 \Bigr) 
 = \varcharf\entvar - \frac12 \Bigl(
   \entvar\incu - \ent(\incu)
   \pm \frac{1}{\lambda} \bigl(\entvar\phi(\incu) - \entflux(\incu)\bigr)
 \Bigr) \\
 &= \entvar \Bigl(
  \varcharf - \frac12 \Bigl(\incu \pm \frac{1}{\lambda}\phi(\incu) \Bigr)
 \Bigr) + \frac12 \Bigl(
 \ent(\incu) \pm \frac{1}{\lambda}\entflux(\incu)
 \Bigr) \\
 &= \entvar (\varcharf-\hpm(\incu)) + \frac12 \Bigl(
 \ent\pm\frac1\lambda\entflux
 \Bigr) \Bigl(
 (\hpm)^{-1}(\varcharf)
 \Bigr) = \kinent(\varcharf),
\end{align*}
and the proof is established. 
\end{proof}

\subsection{Numerical entropies estimates}

In this section, we construct a numerical entropy and the corresponding numerical entropy-flux for the numerical scheme and we prove the dissipation of this numerical entropy.
Inspired by \cite{LatSer:2001:0}, let us define the numerical entropies for the numerical scheme by 
\begin{align}\label{defnument}
\Ejnpd &= \kinentp(\fpjnpd) + \kinentm(\fmjnpd),
& j\in\Z, \ n\in\N,\\
\Qjnpd &= \lambda \kinentp(\fpjnpd) - \lambda \kinentm(\fmjnpd[j+1]),
& j\in\Z, \ n\in\N,
\end{align}
where $\lambda=\dx/\dt$.
Note that the numerical entropies are defined for each time step after the relaxation phase and before the transport phase. This is essential in order to obtained the estimates.
Let us now define the numerical entropy production
\begin{equation}\label{eq:diss_num_ent}
\entprodjn = \frac{\Ejnpd-\Ejnmd}{\dt} + \frac{\Qjnmd-\Qjnmd[j-\onehalf]}{\dx},
\qquad j\in\Z, \ n\in\N^*.
\end{equation}
We begin by proving that the entropy production has a sign.

\begin{proposition}
\label{thm:prop_eq:diss_num_ent}
Under assumptions \ref{hyp:s} and \ref{hyp:CFL}, we have
\begin{equation*}
\entprodjn\leq 0,\qquad j\in\Z, \ n\in\N^*.
\end{equation*}
\end{proposition} 

\begin{proof}
We develop \(\entprodjn\). Since we have \(\lambda=\dx/\dt\), we obtain
\begin{equation*}
\entprodjn = \frac{\kinentp(\fpjnpd)-\kinentp(\fpjnmd[j-1])}{\dt}
  + \frac{\kinentm(\fmjnpd)-\kinentm(\fmjnmd[j+1])}{\dt}.
\end{equation*}

By using one time step of the scheme, we have 
\begin{equation*}
\fpmjnpd[j] = 
\fpmjnmd[j\mp1]
  \pm \frac{s}{2\lambda}\bigl(\phi(\ujnmd[j])-\vjnmd[j]\bigr)
=
\fpmjnmd[j\mp1]
  \pm \frac{s}{2\lambda(1-s)}\bigl(\phi(\ujnpd[j])-\vjnpd[j]\bigr),
\end{equation*}
the convexity of \(\kinent\) implies that
\begin{equation*}
\kinent(\fpmjnpd[j]) \mp \frac{s}{2\lambda(1-s)}
\bigl(\phi(\ujnpd[j])-\vjnpd[j]\bigr)\kinentprime(\fpmjnpd[j])
\leq \kinent(\fpmjnmd[j\mp1]).
\end{equation*}
We have thus that
\begin{equation*}
\frac{\kinent(\fpmjnpd[j])-\kinent(\fpmjnmd[j\mp1])}{\dt}
\leq \pm \frac{s}{2\lambda(1-s)}
\bigl(\phi(\ujnpd[j])-\vjnpd[j]\bigr)\kinentprime(\fpmjnpd[j]).
\end{equation*}
Hence we get
\begin{equation}\label{eq:eq_est_ent}
\entprodjn \leq
 \frac{s}{2\lambda(1-s)}
 \bigl(\phi(\ujnpd[j])-\vjnpd[j]\bigr)
 \bigl(\kinentpprime(\fpjnpd[j])-\kinentmprime(\fmjnpd[j]) \bigr).
\end{equation}

Now from 
\begin{equation*}
\fpmjnpd[j] = \hpm(\ujnpd[j])
\pm \frac{1}{2\lambda} \bigl(\vjnpd[j]-\phi(\ujnpd[j])\bigr)
\end{equation*}
it follows that 
\begin{equation*}
\kinentprime(\fpmjnpd[j]) = 
\kinentprime(\hpm(\ujnpd[j]))
\pm \frac{1}{2\lambda} \bigl(\vjnpd[j]-\phi(\ujnpd[j])\bigr)
\kinentsecond(\xi^{\pm, n+\onehalf}_j),
\end{equation*}
where \(\xi^{\pm, n+\onehalf}_j\) lie between \(\fpmjnpd[j]\) and \(\hpm(\ujnpd[j])\). Combining with \eqref{eq:eq_est_ent} and using Lemma~\ref{th:lm_kinent}, we conclude that
\begin{equation*}
\entprodjn \leq -\frac{s}{4\lambda^2(1-s)}
\bigl(\phi(\ujnpd[j])-\vjnpd[j]\bigr)^2
\bigl(
\kinentsecond[\plus](\xi^{\plus, n+\onehalf}_j)
+
\kinentsecond[\moins](\xi^{\moins, n+\onehalf}_j)
\bigr) \leq 0,
\end{equation*}
and the proof is established.
\end{proof}

\section{Convergence towards the entropic solution}
\label{sec:conv_entropysol}

In this section we establish the final convergence result, by using the discrete entropy estimates obtained at the previous section. We will prove that the weak solution 
of the nonlinear Cauchy problem (\refeq{eq:s1_edp}-\refeq{eq:s1_ic}) given by Theorem \ref{th:theo_conv}, obtained as the limit \(\ulim\) of the numerical scheme,
 is indeed the unique entropic solution of the Cauchy problem (\refeq{eq:s1_edp}-\refeq{eq:s1_ic}). 

\begin{theorem}[convergence result]
\label{theo_sol_ent}
Let \(\incu\) be a weak solution of the Cauchy problem (\refeq{eq:s1_edp}-\refeq{eq:s1_ic}) given by Theorem~\ref{th:theo_conv}. Then we have that \(\incu\) is the unique entropic solution of (\refeq{eq:s1_edp}-\refeq{eq:s1_ic}).
\end{theorem}

The main tools of the proof are classical in the context of the finite volume schemes \cite{GodRav:1991:0}, 
but the choice of the numerical entropies \eqref{defnument} being specific to the \duqd scheme, we present it here.

\begin{proof}
Let \((\ent, \entflux)\) be an entropy-entropy-flux pair for \eqref{eq:s1_edp}, with \(\ent\) strictly convex and let \(\entprodjn\) be defined by \eqref{eq:diss_num_ent}. Let us also consider \(\psi\in\mathcal{D}(]0,+\infty[\times\R)\), \(\psi\geq0\) and put 
\begin{gather*}
\psinj = 
\psi(\tn, \xj),
\quad j\in\Z, \ n\in\N,
\\
\psidelta(\vart, \varx) =
\sum_{n\in\N} \sum_{j\in\Z}
\indicatrice_{[\tn,\tn[n+1])}(\vart)
\indicatrice_{[\xj,\xj[j+1])}(\varx) \psinj.
\end{gather*}
The result of Proposition \ref{thm:prop_eq:diss_num_ent} implies that
\begin{equation*}
\dt\dx\entprodjn\psinj \leq 0,
\qquad j\in\Z, \ n\in\N^*,
\end{equation*}
and by summing over \(n\in\N^*\) and over \(j\in\Z\), we get
\begin{equation*}
\dt\dx \sum_{j\in\Z} \sum_{n\in\N^*} \psinj \left(
\frac{\Ejnpd-\Ejnmd}{\dt} + \frac{\Qjnmd-\Qjnmd[j-\onehalf]}{\dx}
\right)\le0. 
\end{equation*}
We now do a discrete integration  by parts. We obtain
\begin{equation*}
\dt\dx \left(
\sum_{j\in\Z}\sum_{n\in\N}\Ejnpd\frac{\psinj-\psinpj}{\dt}
+
\sum_{j\in\Z} \sum_{n\in\N^*} \Qjnmd \frac{\psinj-\psinjp}{\dx}
\right)\le0.
\end{equation*}
By using the definition \eqref{defnument} of the numerical entropies \(\Ejnpd\) and \(\Qjnmd\), the above inequality writes as
\begin{multline*}
\sum_{j\in\Z}\sum_{n\in\N}
\int_{\xj}^{\xj[j+1]}\!\!\!
\int_{\tn}^{\tn[n+1]}\!\!\!
\Bigl(e^{+}\!(\fpjnpd)+e^{-}\!(\fmjnpd)\Bigr)
\frac{\psi(\tn,\xj)-\psi(\tn[n+1],\xj)}{\dt}\ddd\vart\ddd\varx\\
+
\sum_{j\in\Z}\sum_{n\in\N^*}
\int_{\xj}^{\xj[j+1]}\!\!\!
\int_{\tn}^{\tn[n+1]}\!\!\!\!\!\!
\lambda\Bigl(e^{+}\!(\fpjnpd)-e^{-}\!(\fmjnpd[j+1])\Bigr)
\frac{\psi(\tn,\xj)-\psi(\tn,\xj[j+1])}{\dx}\ddd\vart\ddd\varx\leq0.
\end{multline*}
We now use that \(\fpjnp=\fpjnpd[j-1]\), \(\fmjnp=\fmjnpd[j+1]\), for all \(j\in\Z\) and \(n\in\N\). We obtain
\begin{multline*}
\sum_{j\in\Z}\sum_{n\in\N}
\int_{\xj}^{\xj[j+1]}\!\!\!
\int_{\tn}^{\tn[n+1]}\!\!\!
\Bigl(e^{+}\!(\fpjnpd)+e^{-}\!(\fmjnpd)\Bigr)
\frac{\psi(\tn,\xj)-\psi(\tn[n+1],\xj)}{\dt}\ddd\vart\ddd\varx\\
+
\sum_{j\in\Z}\sum_{n\in\N^*}
\int_{\xj}^{\xj[j+1]}\!\!\!
\int_{\tn}^{\tn[n+1]}\!\!\!\!\!\!
\lambda\Bigl(e^{+}\!(\fpjnpd)-e^{-}\!(\fmjnpd[j+1])\Bigr)
\frac{\psi(\tn,\xj)-\psi(\tn,\xj[j+1])}{\dx}\ddd\vart\ddd\varx\\
=\int_{0}^{+\infty}\int_\R \big(e^+(\fpdelta(t+\dt,x+\dx))+e^-(\fmdelta(t+\dt,x-\dx))\big)\\
\times\Bigg(\frac{\psidelta(t,x)-\psidelta(t+\dt,x)}{\dt}\Bigg)\ddd\vart\ddd\varx\\
+\int_{\dt}^{+\infty}\int_\R \big(\lambda e^+(\fpdelta(t+\dt,x+\dx))-\lambda e^-(\fmdelta(t+\dt,x))\big)\\
\times\Bigg(\frac{\psidelta(t,x)-\psidelta(t,x+\dx)}{\dx}\Bigg)\ddd\vart\ddd\varx\le0. 
\end{multline*}
Following the definition of \((e^+,e^-)\) and since 
$(\fpdelta,\fmdelta)$ tends to $(h^+(u),h^-(u))$ when $\dt$ and $\dx$ vanish,
by passing to the limit as \(\dt,\dx\to0\), we obtain
\begin{equation*}
\int_{0}^{+\infty}\int_\R \left(\eta(u)\frac{\partial\psi}{\partial t}+
q(u)\frac{\partial\psi}{\partial x}\right)\ddd\varx\ddd\vart\geq 0,
\end{equation*}
and the result of the theorem follows. 
\end{proof}

\section{Numerical illustrations}\label{sec:num_ill}

In this section, numerical simulations are given in order to illustrate the theoretical results of the previous sections. Two models are investigated: the advection equation with a constant velocity and the Burgers equation, both considering regular then discontinuous initial conditions. In particular, the numerical rate of convergence and the entropy production are computed in these cases.

Two initial conditions are systematically chosen: first a regular function 
\begin{equation}\label{eq:numuoreg}\tag{$I_A$}
 \incu^0(\varx) = \left\lbrace
\begin{aligned}
& 0 & \text{ if }\quad & \phantom{x_L{-}\delta\leq \;}\varx \leq x_L{-}\delta,\\
 &\tfrac{1}{2} + \tfrac{(\varx-x_L)(3\delta^2-(\varx-x_L)^2)}{4\delta^3}
 &\text{ if }\quad & x_L{-}\delta \leq \varx \leq x_L{+}\delta,\\
 & 1& \text{ if }\quad &x_L{+}\delta \leq \varx \leq x_R{+}\delta,\\
 &\tfrac{1}{2} - \tfrac{(\varx-x_R)(3\delta^2-(\varx-x_R)^2)}{4\delta^3}
 &\text{ if }\quad & x_R{-}\delta \leq \varx \leq x_R{+}\delta,\\
& 0 & \text{ if }\quad &x_R{+}\delta \leq \varx,
\end{aligned}
 \right.
\end{equation}
and second a discontinuous function
\begin{equation}\label{eq:numuodis}\tag{$I_B$}
 \incu^0(\varx) = \left\lbrace
\begin{aligned}
& 0 & \text{ if }\quad & \phantom{x_L\leq \;}\varx \leq x_L,\\
 & 1& \text{ if }\quad &x_L \leq \varx \leq x_R,\\
& 0 & \text{ if }\quad &x_R \leq \varx.
\end{aligned}
 \right.
\end{equation}
Numerically, we took
\(x_L=1/4\), \(x_R=3/4\), \(\delta=1/10\).
Note that the second case \eqref{eq:numuodis} can be obtained from the first case \eqref{eq:numuoreg} in the limit \(\delta\) goes to 0.

Concerning the choice of the relaxation parameter, asumption~\ref{hyp:s} only imposes that $s$ has to be a constant value lying in $(0,1]$. We have therefore chosen several values $$s\in\lbrace 0.1, 0.2, 0.5, 0.7, 0.9, 1.0 \rbrace$$ to illustrate the asymptotic convergence of the solution.

\subsection{Constant velocity advection equation}

The first studied model corresponds to the advection equation at velocity \(\vitessetransport\)
\begin{equation}\label{eq:numsystr}
 \left\lbrace
\begin{aligned}
& \drondt \incu(\vart, \varx) + \vitessetransport \drondx \incu(\vart, \varx) = 0, &&\varx\in\R, \vart>0,\\
&\incu(0,\varx) = \incu^0(\varx), &&\varx\in\R,
\end{aligned}
 \right.
\end{equation}
where the velocity \(\vitessetransport\) is taken equal to \(0.75\). 
The exact solution is well known for both initial conditions \eqref{eq:numuoreg} and \eqref{eq:numuodis} and reads \(\incu(\vart, \varx) = \incu^0(\varx - \vitessetransport\vart)\).

\begin{figure}[ht]
\centering
\includegraphics[width=0.49\linewidth]{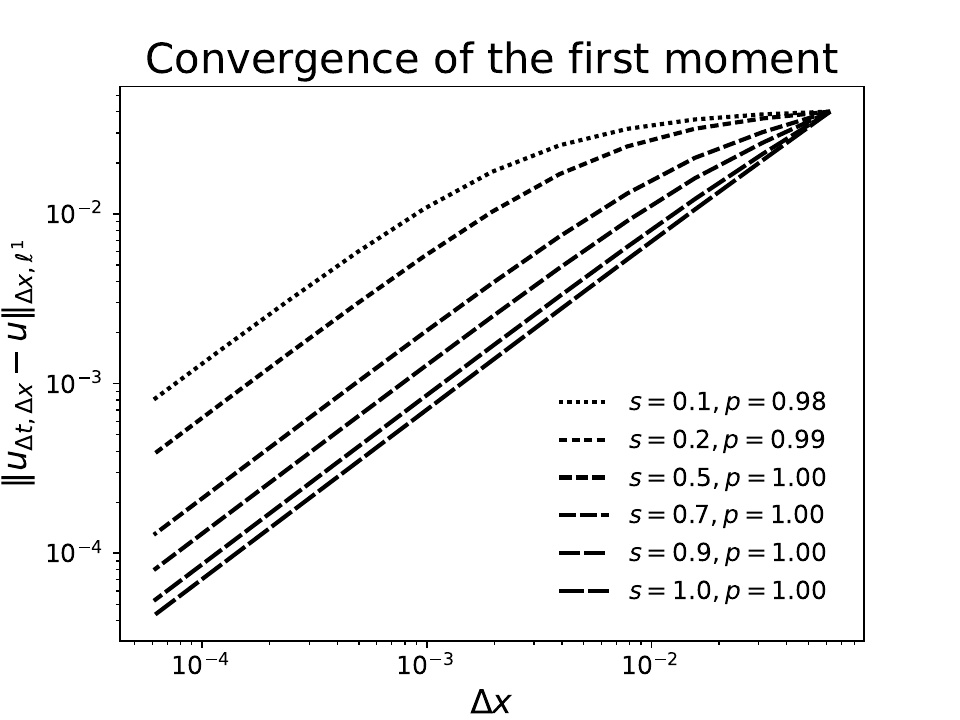}
\hfill
\includegraphics[width=0.49\linewidth]{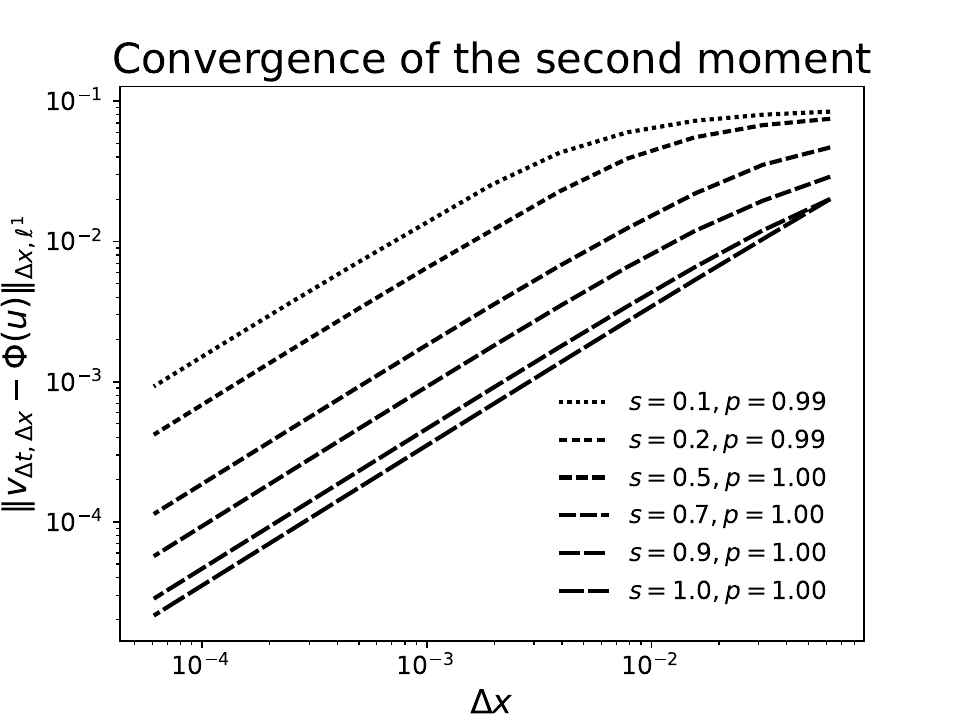}
\caption{Initial condition \eqref{eq:numuoreg}. Error in norm \(\ell^1\) of the numerical solution of the advection equation \eqref{eq:numsystr} according to the space steps \(\dx\) for several relaxation parameter values \(s\in\lbrace 0.1,0.2,0.5,0.7,0.9,1.0 \rbrace\). \(p\) is the convergence rate.}
\label{fig:numtrconvreg}
\end{figure}

\begin{figure}[ht]
\centering
\includegraphics[width=0.49\linewidth]{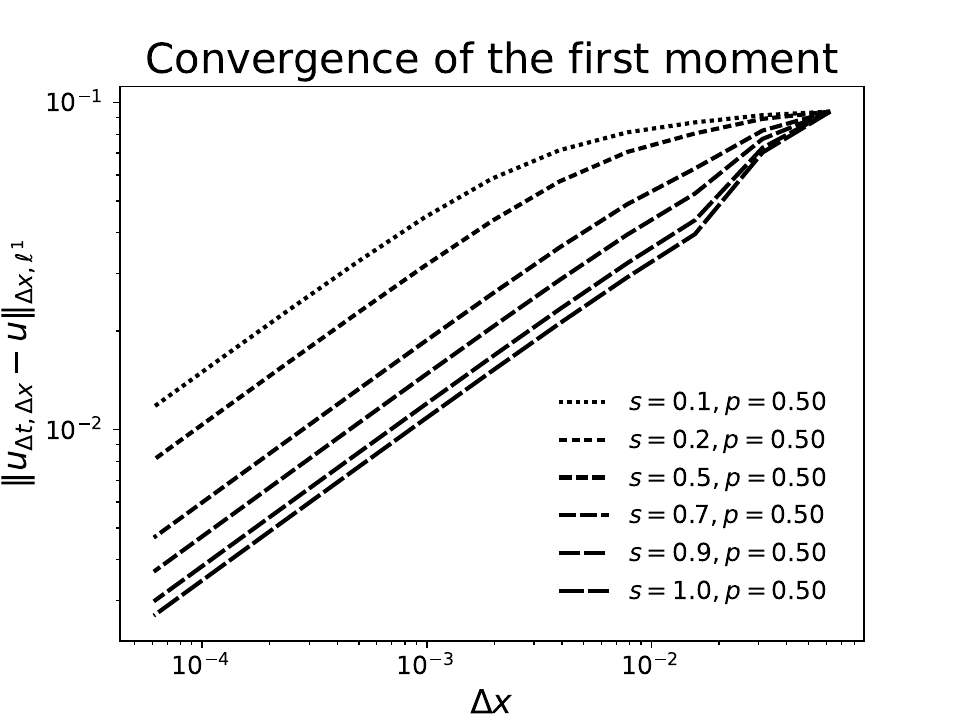}
\hfill
\includegraphics[width=0.49\linewidth]{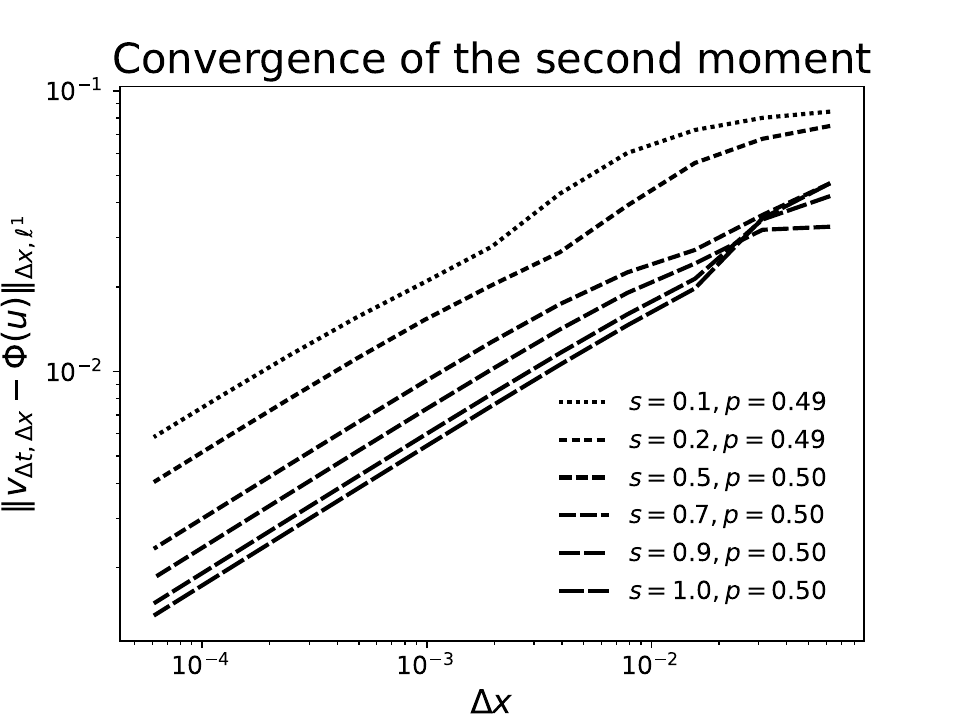}
\caption{Initial condition \eqref{eq:numuodis}. Error in norm \(\ell^1\) of the numerical solution of the advection equation \eqref{eq:numsystr} according to the space steps \(\dx\) for several relaxation parameter values \(s\in\lbrace 0.1,0.2,0.5,0.7,0.9,1.0 \rbrace\). \(p\) is the convergence rate.}
\label{fig:numtrconvdis}
\end{figure}

The convergence rates of the numerical solution can be read in Figure~\ref{fig:numtrconvreg} for the regular solution and in Figure~\ref{fig:numtrconvdis} for the discontinuous solution: the error on the first moment \(\incun\) and on the second moment \(\incvn\) converges towards 0 with a convergence rate equal to \(1\) for regular solutions and equal to \(0.5\) for discontinuous solutions. As expected, the error is even larger when the relaxation parameter \(s\) is small, this numerical results staying true for \(s\) lying in \([1,2]\) as observed in \cite{Gra201400} for the first moment in other test cases.

\begin{figure}[ht]
\includegraphics[width=0.49\linewidth]{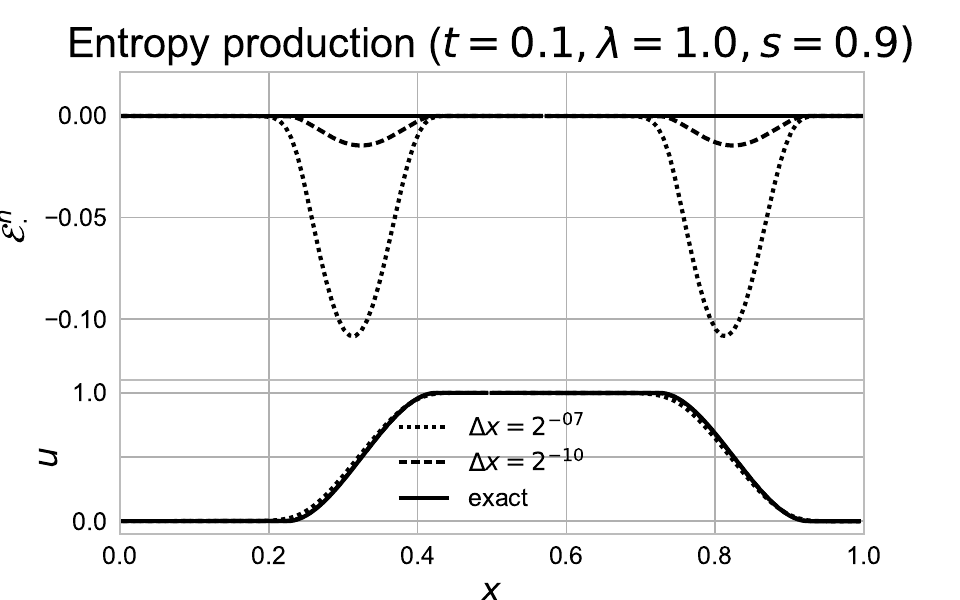}
\hfill
\includegraphics[width=0.49\linewidth]{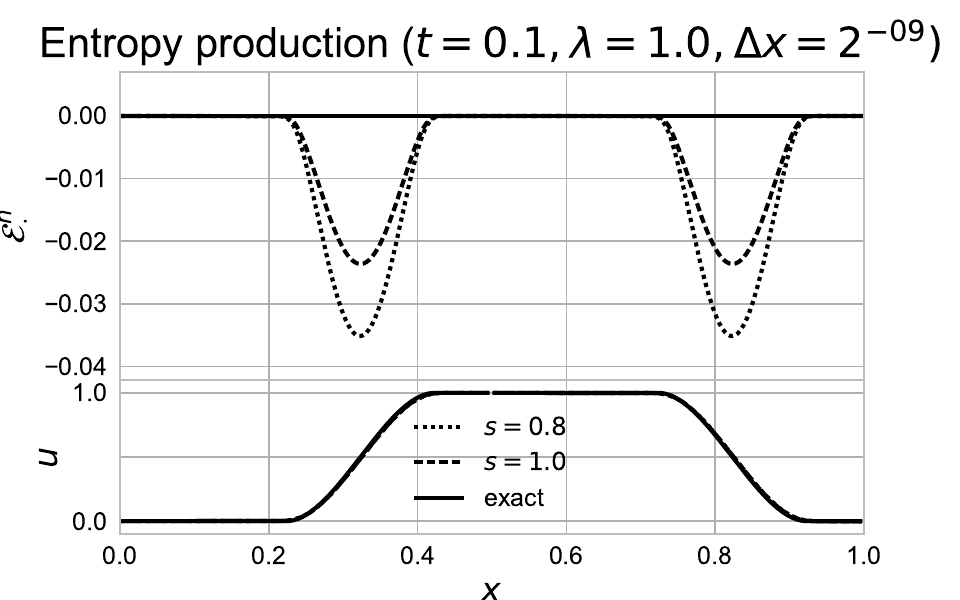}
\caption{Local in space numerical entropy production of the advection equation with initial condition \eqref{eq:numuoreg}for varying space steps \(\dx\) at left and for varying relation parameters \(s\) at right. The numerical and exact solutions at \(t=0.1\) are shown bellow and the corresponding entropy productions above.}
\label{fig:numtrregentprodloc}
\end{figure}

\begin{figure}[ht]
\includegraphics[width=0.49\linewidth]{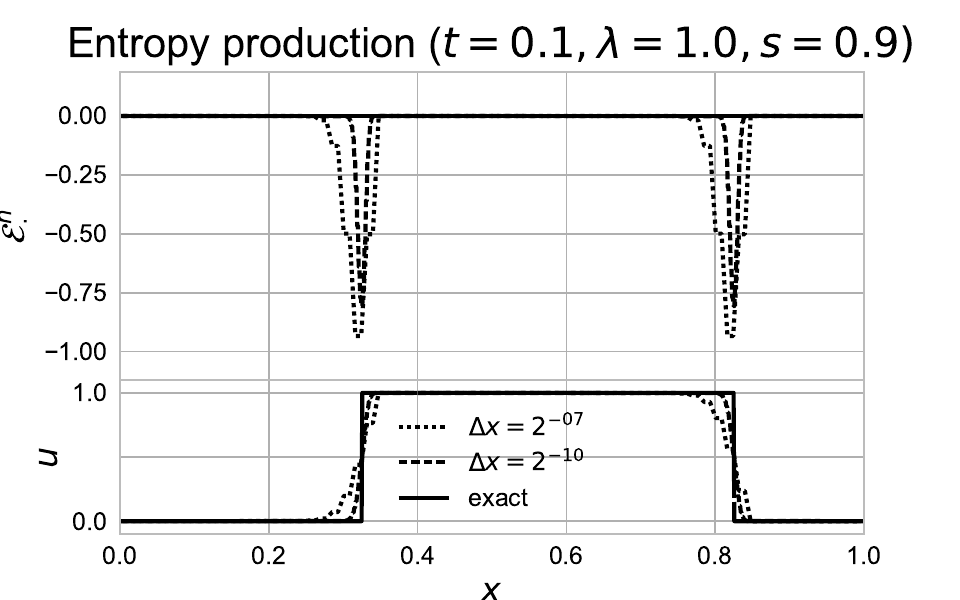}
\hfill
\includegraphics[width=0.49\linewidth]{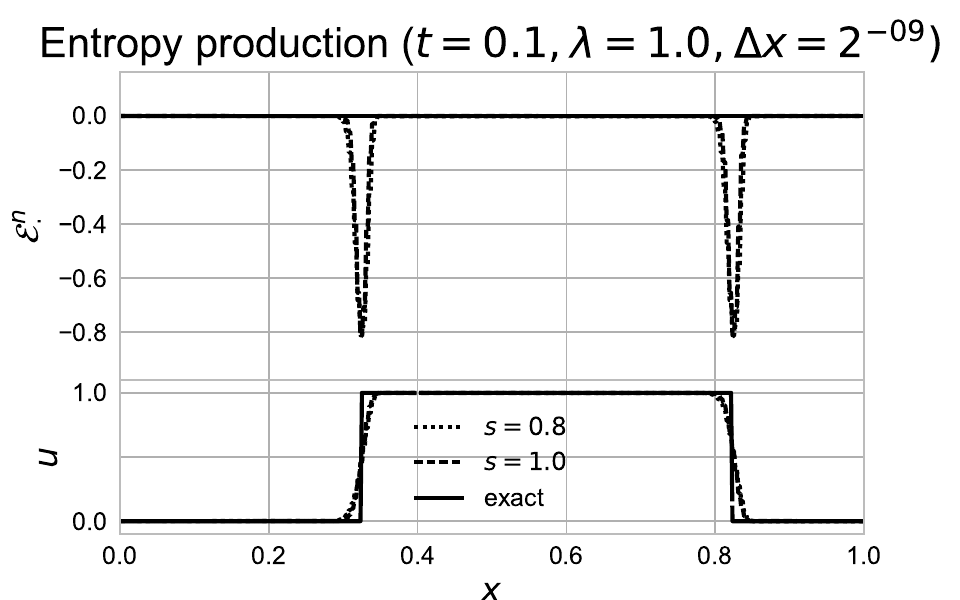}
\caption{Local in space numerical entropy production of the advection equation with initial condition \eqref{eq:numuodis}for varying space steps \(\dx\) at left and for varying relation parameters \(s\) at right. The numerical and exact solutions at \(t=0.1\) are shown bellow and the corresponding entropy productions above.}
\label{fig:numtrdisentprodloc}
\end{figure}

We now present the results on the numerical entropy production. As the system~\eqref{eq:numsystr} is linear, the exact entropy production is zero. 
In Figure~\ref{fig:numtrregentprodloc}, the local in space numerical entropy production \(\entprodjn\), \(j\in\Z\), is shown for the regular initial condition \eqref{eq:numuoreg}. This production is negative as proved in Proposition~\ref{thm:prop_eq:diss_num_ent}, located where the solution is not constant, and is smaller for small space step \(\dx\) and for large relaxation parameter \(s\).
In Figure~\ref{fig:numtrdisentprodloc}, the entropy production is shown for the discontinuous initial condition \eqref{eq:numuodis}. This production is also negative and located on the discontinuities. The \(\ell^1\)-norm is decreasing when the space step \(\dx\) is lower or when the relaxation parameter \(s\) is larger.

\subsection{Burgers equation}

The second studied model corresponds to the Burgers equation
\begin{equation}\label{eq:numsysB}
 \left\lbrace
\begin{aligned}
& \drondt \incu(\vart, \varx) + \drondx (\tfrac{1}{2}\incu^2(\vart, \varx)) = 0, &&\varx\in\R, \vart>0,\\
&\incu(0,\varx) = \incu^0(\varx), &&\varx\in\R.
\end{aligned}
 \right.
\end{equation}
The exact solution for both initial conditions can be computed: for the regular function \eqref{eq:numuoreg} with the theory of the characteristics until the shock appears and for the discontinuous function \eqref{eq:numuodis} with the combination of a shock and a rarefaction wave. 

\begin{figure}[ht]
\centering
\includegraphics[width=0.49\linewidth]{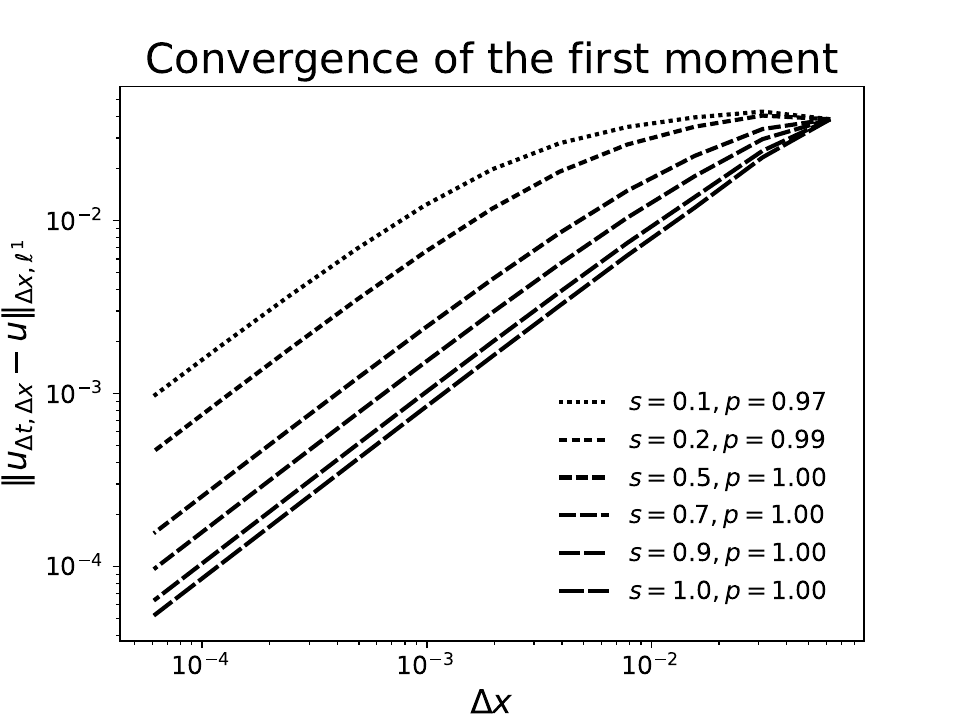}
\hfill
\includegraphics[width=0.49\linewidth]{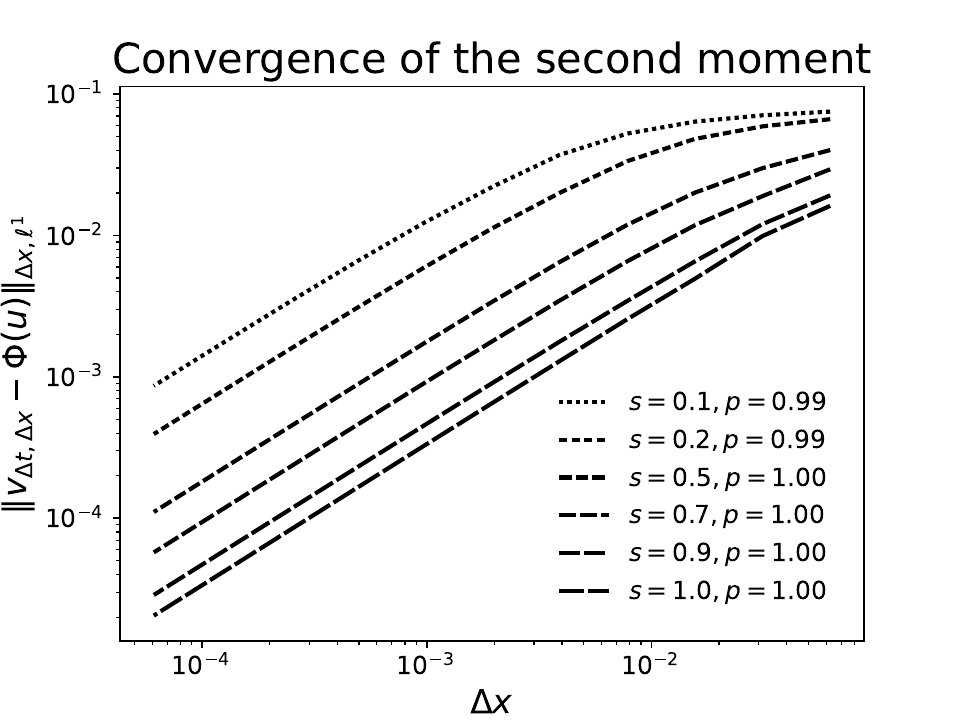}
\caption{Initial condition \eqref{eq:numuoreg}. Error in norm \(\ell^1\) of the numerical solution of the Burgers equation \eqref{eq:numsysB} according to the space steps \(\dx\) for several relaxation parameter values \(s\in\lbrace 0.1,0.2,0.5,0.7,0.9,1.0 \rbrace\). \(p\) is the convergence rate.}
\label{fig:numBuconvreg}
\end{figure}

\begin{figure}[ht]
\centering
\includegraphics[width=0.49\linewidth]{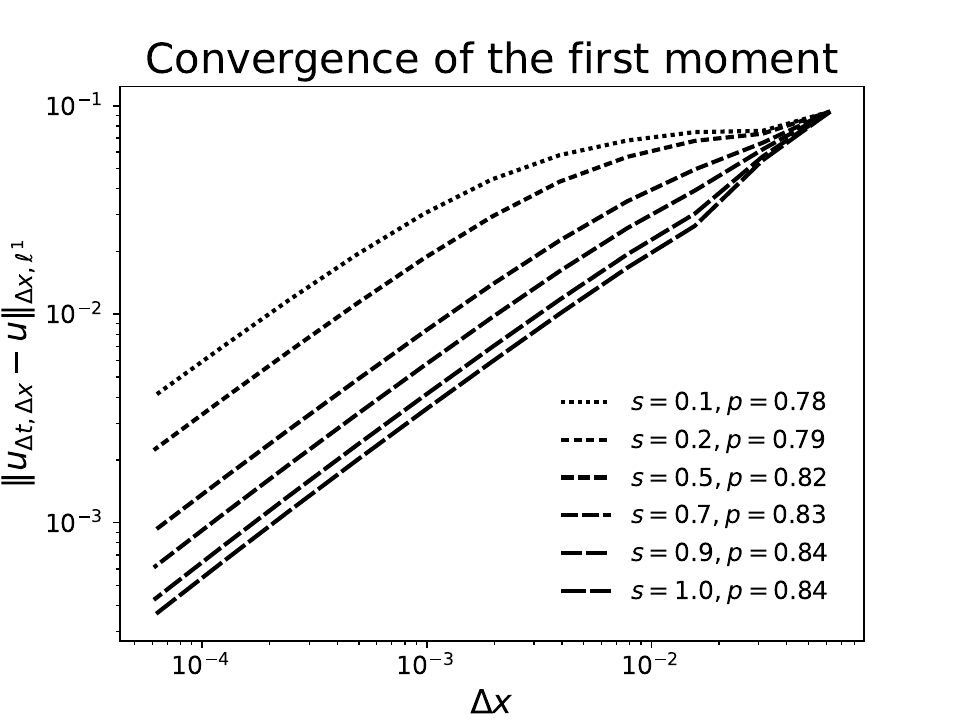}
\hfill
\includegraphics[width=0.49\linewidth]{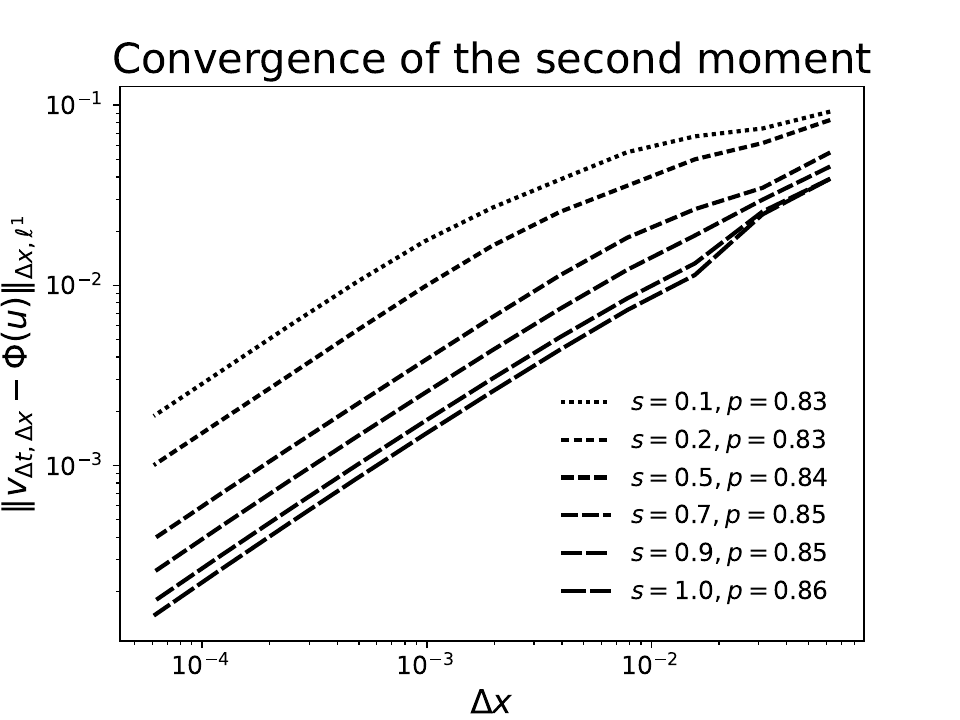}
\caption{Initial condition \eqref{eq:numuodis}. Error in norm \(\ell^1\) of the numerical solution of the Burgers equation \eqref{eq:numsysB} according to the space steps \(\dx\) for several relaxation parameter values \(s\in\lbrace 0.1,0.2,0.5,0.7,0.9,1.0 \rbrace\). \(p\) is the convergence rate.}
\label{fig:numBuconvdis}
\end{figure}

The convergence rates of the numerical solution can be read in Figure~\ref{fig:numBuconvreg} for the regular solution and in Figure~\ref{fig:numBuconvdis} for the discontinuous solution: the error on the first moment \(\incun\) and on the second moment \(\incvn\) converges towards 0.
For regular solutions, the convergence rate is equal to \(1\) and, for discontinuous solutions, it is about~\(0.8\). As expected, the error is even larger when the relaxation parameter \(s\) is small, this numerical results staying true for \(s\) lying in \([1,2]\) as observed in \cite{Gra201400} for the first moment in other test cases.
Note that we cannot explain the better convergence rate of this nonlinear equation for discontinuous solutions as previously shown in \cite{Gra201400}. 

\begin{figure}[ht]
\includegraphics[width=0.49\linewidth]{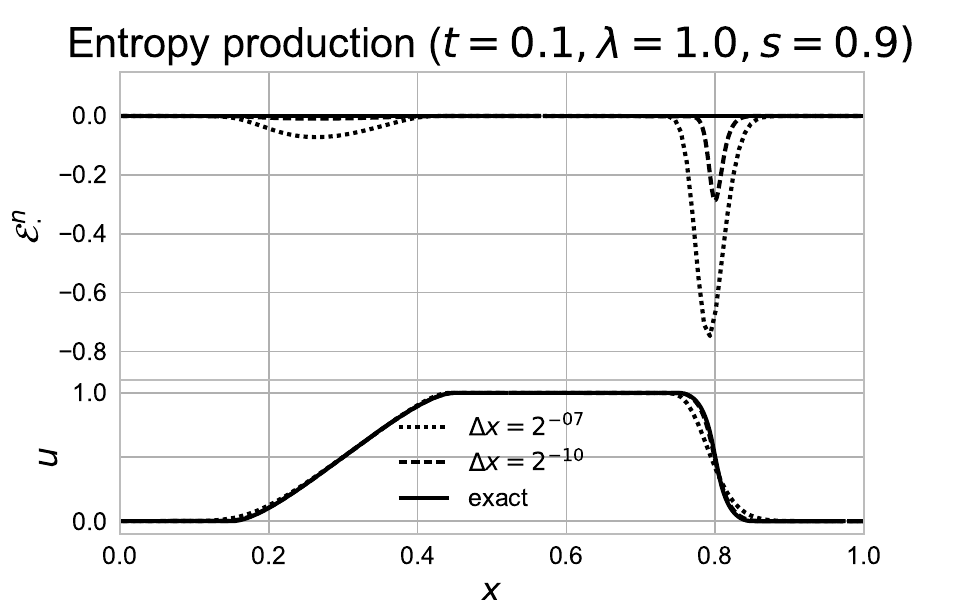}
\hfill
\includegraphics[width=0.49\linewidth]{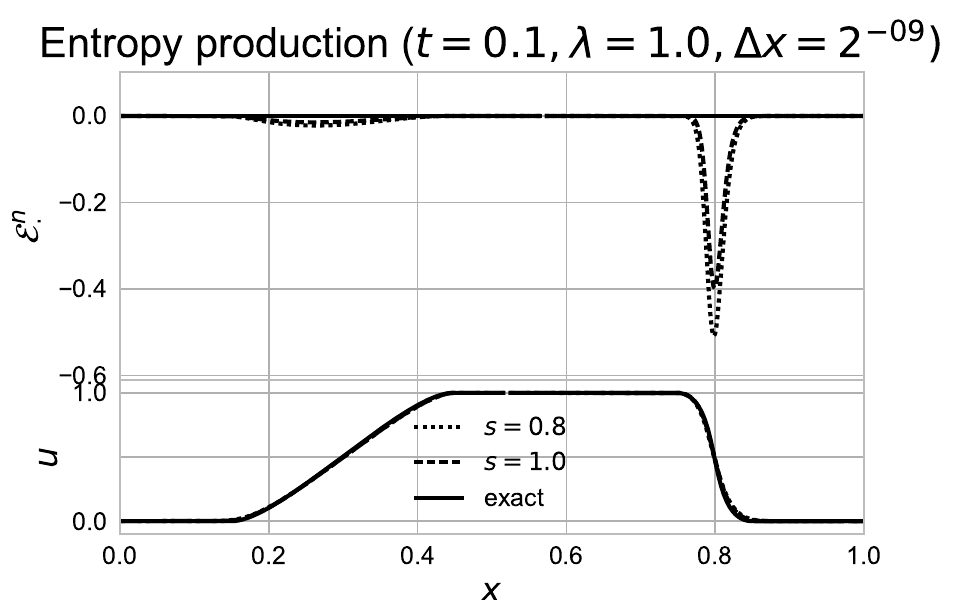}
\caption{Local in space numerical entropy production of the Burgers equation with initial condition \eqref{eq:numuoreg}for varying space steps \(\dx\) at left and for varying relation parameters \(s\) at right. The numerical and exact solutions at \(t=0.1\) are shown bellow and the corresponding entropy productions above.}
\label{fig:numBuregentprodloc}
\end{figure}

\begin{figure}[ht]
\includegraphics[width=0.49\linewidth]{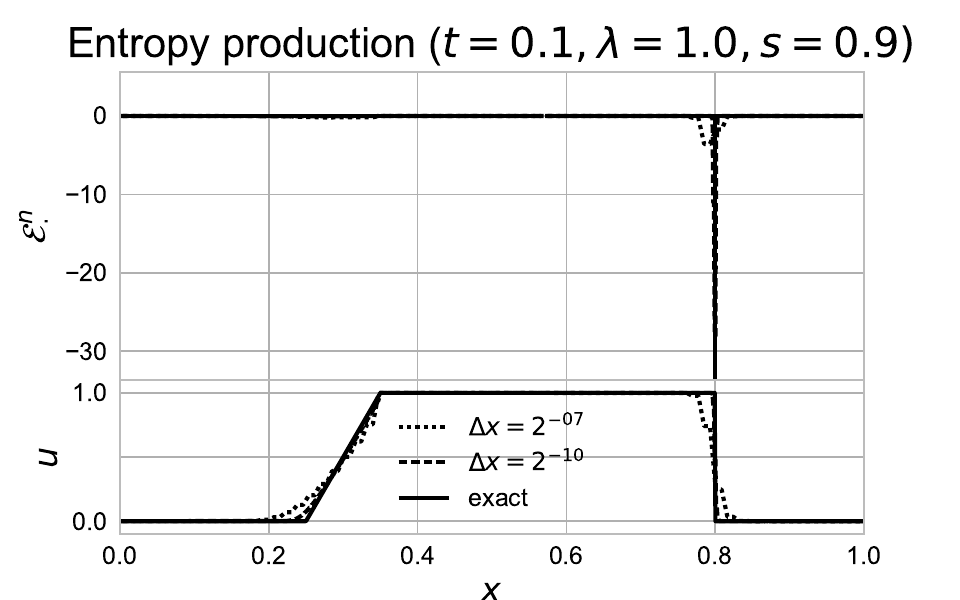}
\hfill
\includegraphics[width=0.49\linewidth]{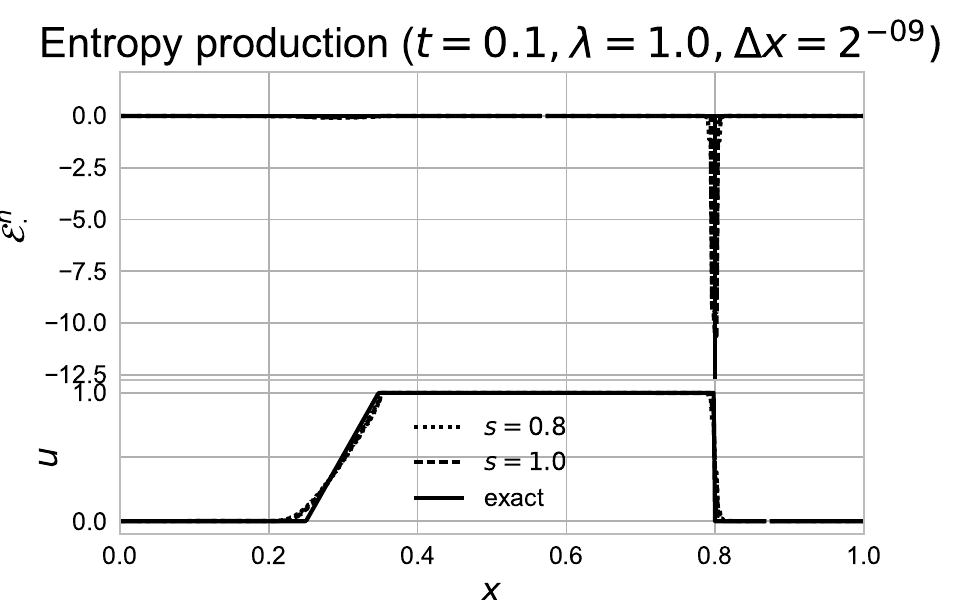}
\caption{Local in space numerical entropy production of the Burgers equation with initial condition \eqref{eq:numuodis}for varying space steps \(\dx\) at left and for varying relation parameters \(s\) at right. The numerical and exact solutions at \(t=0.1\) are shown bellow and the corresponding entropy productions above.}
\label{fig:numBudisentprodloc}
\end{figure}

We now present the results on the numerical entropy production. For the system~\eqref{eq:numsysB}, the exact entropy production is zero but on the decreasing discontinuities where it is a Dirac measure multiplied by a negative value. 
In Figure~\ref{fig:numBuregentprodloc}, the local in space numerical entropy production \(\entprodjn\), \(j\in\Z\), is shown for the regular initial condition \eqref{eq:numuoreg}. This production is negative as proved in Proposition~\ref{thm:prop_eq:diss_num_ent}, located where the solution is not constant, and is smaller for small space step \(\dx\) and for large relaxation parameter \(s\).
In Figure~\ref{fig:numBudisentprodloc}, the entropy production is shown for the discontinuous initial condition \eqref{eq:numuodis}. This production is also negative and essentially located on the discontinuity (a smaller contribution is located on the rarefaction wave, this contribution being decreasing when the space step \(\dx\) is smaller or when the relaxation parameter \(s\) is larger).

\subsection{Concluding remarks}

These numerical illustrations confirm the theoretical results of the convergence of the numerical solution for the \(\ell^1\)-norm for both moments \(\incu\) and \(\incv\). Moreover, the discrete entropy production is non-positive and seems to converge towards the exact entropy production.

Concerning the numerical convergence rate, we observe that it is equal to 1 for a smooth initial condition, as long as the solution remains regular enough. By putting the two equations \eqref{eq:s1_edp} and \eqref{eq:eqeq} (the target equation and the equivalent equation) side by side, this value is compatible with a truncation error of order 1 in time. 
In the case of a solution with regularity breaks, we notice a lower convergence rate.
One of the authors already observed these results without rigorous proof in \cite{Gra201400}.

\section{Conclusion}

In this paper, we prove, by using techniques based on finite volume methods---in particular on relaxation methods---, a convergence result for a non-linear hyperbolic one-dimensional scalar conservation law. More precisely, we prove the convergence of the numerical solutions of the \duqd scheme towards the unique entropy solution of the scalar conservation law. We prove in addition a numerical entropy estimation. Our results are based on the convexity properties of the scheme.

Future works could be dedicated to extend the present results to a wide range of relaxation parameters. Moreover, other popular lattice Boltzmann schemes, like a scheme with three velocities in one space dimension, should also be studied. Of course, general results concerning the convergence of more complex lattice Boltzmann schemes remains an open question. From a numerical point of view, we noticed different convergence rates according to whether the solution of the conservation law is discontinuous or not. The convergence rate of the scheme is also a subject to be investigated.

\section*{Acknowledgments}
The authors are grateful to Fr\'ed\'eric Lagouti\`ere and to Jean-Fran\c cois Babadjian for their explanations concerning bounded variation functions. The simulations have been obtained with the \texttt{python} package \text{pylbm} (\url{https://pylbm.readthedocs.io/en/latest/}).

\begin{appendix}

  \section{Proof of lemma \ref{lem_VTu0}}
  \label{ap_lem_VTu0}
  
  Using the definition of \(\ujz\) given in assumption~\ref{hyp:initeq}, we have that 
  \begin{multline*}
  \tv(\un[0]) = \frac{1}{\dx} \sum_{j\in\Z}
    \Bigg\vert  
      \int_{\xj}^{\xj[j+1]} \big(
        \incu^0(\varx+\dx) - \incu^0(\varx)
      \big) \ddd\varx
    \Bigg\vert  \\
  \leq
  \frac{1}{\dx} \sum_{j\in\Z} \int_{\xj}^{\xj[j+1]}
    \big\vert  \incu^0(\varx+\dx) - \incu^0(\varx) \big\vert 
    \ddd\varx
  =
  \frac{1}{\dx} \int_\R \big\vert 
    \incu^0(\varx+\dx) - \incu^0(\varx)
  \big\vert  \ddd\varx.
  \end{multline*}
  We will then prove that
  \begin{equation}\label{theovtu0}
  \frac{1}{\dx} \int_\R \big\vert 
    \incu^0(\varx+\dx) - \incu^0(\varx)
  \big\vert  \ddd\varx
  \leq \tv_\R(\incu^0).
  \end{equation}
  Suppose first that \(\incu^0\in \mathcal{C}^1(\R)\). In this case we have that 
  \begin{equation*}
  \tv_\R(\incu^0) = \int_\R \bigl\vert 
    (\incu^0)^{\prime}(\varx) 
  \bigr\vert  \ddd\varx
  \end{equation*}
  and we get
  \begin{multline*}
  \frac{1}{\dx} \int_\R \bigl\vert 
    \incu^0(\varx+\dx) - \incu^0(\varx)
  \bigr\vert \ddd\varx =
  \frac{1}{\dx} \int_\R \Bigl\vert  \int_{\varx}^{\varx+\dx}
  (\incu^0)^{\prime}(\vary) \ddd\vary \Bigr \vert \ddd\varx \\
  \leq
  \frac{1}{\dx}\int_\R\int_\R\bigl\vert 
  (\incu^0)^{\prime}(\vary)
  \bigr\vert  \indicatrice_{[\varx,\varx+\dx]}(\vary)\ddd\varx \ddd\vary
  =
  \frac{1}{\dx}\int_\R \bigl\vert 
    (\incu^0)^{\prime}(\vary)
  \bigr\vert  \dx \ddd\vary = \tv_\R(\incu^0).
  \end{multline*}
  We now prove \eqref{theovtu0} for general \(\incu^0\in BV(\R)\). To do so, we use the following result (see \cite{AmbFusPal:2000}): for \(f\in \BV(\R)\), there exists a sequence \((f_n)_n\in C^\infty(\R)\) such that \(f_n\to f\) in \(\Lunloc(\R)\) and \(\tv_\R(f_n)\to\tv_\R(f)\).
  
  We consider then such a sequence \((f_n)_n\) that converges towards \(\incu^0\). We have, for all \(M>0\),
  \begin{equation}\label{theovtu0:2}
  \frac{1}{\dx}\int_{-M}^M \bigl\vert  
    f_n(\varx+\dx) - f_n(\varx)
  \bigr\vert  \ddd\varx
  \leq
  \frac{1}{\dx} \int_\R \bigl\vert 
    f_n(\varx+\dx) - f_n(\varx)
  \bigr\vert  \ddd\varx
  \leq \tv_\R(f_n),
  \end{equation}
  since \(f_n\in C^\infty(\R)\).
  We have now, on the one hand, that
  \begin{equation*}
  \frac{1}{\dx} \int_{-M}^M \bigl\vert 
    f_n(\varx+\dx) - f_n(\varx)
  \bigr\vert  \ddd\varx \;\xrightarrow[n\to\infty]{} \frac{1}{\dx}\int_{-M}^M\big\vert \incu^0(\varx+\dx)-\incu^0(\varx)\big\vert \ddd\varx,
  \end{equation*}
  since \(f_n\to \incu^0\) in \(\Lunloc(\R)\). On the other hand, \(\tv_\R(f_n)\to\tv_\R(\incu^0)\), so that, by passing \eqref{theovtu0:2} to the limit as \(n\to\infty\), we obtain
  \begin{equation*}
  \frac{1}{\dx} \int_{-M}^M \bigl\vert 
    \incu^0(\varx+\dx) - \incu^0(\varx)
  \bigr\vert  \ddd\varx \leq \tv_\R(\incu^0).
  \end{equation*}
  By letting now \(M\to\infty\), we obtain
  \begin{equation*}
  \frac{1}{\dx} \int_\R \bigl\vert 
    \incu^0(\varx+\dx) - \incu^0(\varx)
  \bigr\vert  \ddd\varx\leq \tv_\R(\incu^0),
  \end{equation*}
  that ends the proof.
  
  \section{Proof of theorem \ref{th:theo_conv}.}\label{ap_proof_conv}
  
  We have to prove that the families of functions \(\udelta\) and \(\vdelta\) are bounded in \(\Linf(\Omega)\cap\BV(\Omega)\) for all bounded open subset \(\Omega\) of \(\R^+{\times}\R\).    
  Proposition \ref{th:proplinfty} implies that both \(\udelta\) and \(\vdelta\) are uniformly bounded in  \(\Linf(\R^+{\times}\R)\).

  Let \(\Omega\) be a bounded open subset of \(\R^+{\times}\R\), such that \(\Omega\subseteq[0,T]\times[-M,M]\), for some \(T>0\) and \(M>0\), and let \(N,J\in\N\) such that 
  \((N{-}1)\dt\leq T < N\dt\), 
  \((J{-}1)\dx\leq M < J\dx\). We have then that 
  \begin{equation*}
  {\tv}_{\Omega}(\udelta)
  \leq 
  \dt \sum_{n=0}^N \sum_{j=-J}^J
    \bigl\vert \ujn[j+1]-\ujn[j]\bigr\vert 
  + \dx \sum_{n=0}^N \sum_{j=-J}^J
    \bigl\vert \ujnp[j]-\ujn[j]\bigr\vert .
  \end{equation*}
  From Proposition \ref{th:propVT} we obtain that
  \begin{equation*}
  \dt \sum_{n=0}^N \sum_{j=-J}^J
    \bigl\vert \ujn[j+1]-\ujn[j] \bigr\vert 
  \leq N\dt\tv_\R(\incu^0)
  \leq (T+\dt)\tv_\R(\incu^0),
  \end{equation*}
  which is bounded as \(\dt\to0\).
  
  On the other hand, as a consequence of Proposition \ref{th:propVTt}, we have 
  \begin{equation*}
    \dx \sum_{n=0}^N \sum_{j=-J}^J
      \bigl\vert \ujnp[j]-\ujn[j] \bigr\vert 
    \leq 
    \frac{\dx}{\dt}(N+1)\dt 
      2\tv_\R(\incu^0)
    \leq 2\lambda(T+\dt)\tv_\R(\incu^0),
  \end{equation*}
  which is also bounded as \(\dt\to0\).
  We obtain for \(\vdelta\) a similar result.
  These estimations imply that the set \(\lbrace(\udelta,\vdelta)\rbrace_{\dt,\dx}\) remains bounded in \(\Linf(\Omega)\cap \BV(\Omega)\), and this is valid for any bounded set \(\Omega\subseteq\R^+\times \R\). The compactness of \(\BV(\Omega)\cap \Linf(\Omega)\) in \(\Lun(\Omega)\) (Helly's theorem) implies that there exists a sub-sequence of \((\udelta,\vdelta)\) and functions \((\ulim,\vlim)\) satisfying the conditions of the Theorem. 
  
  Proposition \ref{th:prop_eq_gap} now implies that
  \begin{equation*}
  \bigl\Vert 
    \phi(\udelta)-\vdelta
  \bigr\Vert _{\Linf([0,T];\Lunloc(\R))}
  \xrightarrow[\dt,\dx\to0]{}0.
  \end{equation*}
  Since  \((\udelta,\vdelta)\longrightarrow(\ulim,\vlim)\),
  we get \(\phi(\ulim)=\vlim\), a.e. \(x\in\R\), for all \(t>0\).
  
  In order to prove that \(\ulim\) is a weak solution of (\refeq{eq:s1_edp}-\refeq{eq:s1_ic}), let us consider \(\monphi\in \Coinf(\R^+\times\R)\) and put 
  \begin{equation*}
  \phinj = \monphi(\tn,\xj),\qquad
  \phidelta(\vart,\varx) = \sum_{n\in\N}\sum_{j\in\Z}
  \phinj\indicatrice_{[\tn,\tn[n+1])}(\vart)
  \indicatrice_{[\xj,\xj[j+1])}(\varx).
  \end{equation*}
  We multiply both sides of \eqref{eq:sc_one_step_con} by 
  \(\dt\dx\phinj\), we sum over \(n\in\N\) and \(j\in\Z\), do a discrete integration by parts, and pass to the limit as \(\dt,\dx\to0\). 
  The first term on the left-hand side reads
  \begin{multline*}
  \dt\dx \sum_{n\geq 1} \sum_{j\in\Z}
    \ujn[j] \frac{\phinmj-\phinj}{\dt}
  - \dt\dx \sum_{j\in\Z} \frac{\ujz\phizj}{\dt}\\
  = \int_{\dt}^{+\infty} \nspaceint \int_\R
    \udelta(\vart,\varx)
    \frac{\phidelta(\vart{-}\dt,\varx) - \phidelta(\vart,\varx)}{\dt}
    \ddd\varx\ddd\vart \\
  -\int_\R \udelta(0,\varx)\phidelta(0,\varx)\ddd\varx \\
  \xrightarrow[\dt,\dx\to0]{}
  -\int_0^{+\infty}\nspaceint\int_\R \ulim(\vart,\varx)
  \drondt\monphi(\vart,\varx)\ddd\varx\ddd\vart
  - \int_\R \incu^0(\varx) \monphi(0,\varx)\ddd\varx.
  \end{multline*}
  By a similar reasoning we get for the second term on the left-hand side 
  \begin{multline*}
  \dt\dx \sum_{n\geq 0}\sum_{j\in\Z}
  \vjn[j]\frac{\phinjm-\phinjp}{2\dx}\\
  =\int_{0}^{+\infty}\nspaceint\int_\R \vdelta(\vart,\varx) \frac{\phidelta(\vart,\varx{-}\dx)-\phidelta(\vart,\varx{+}\dx)}{2\dx}\ddd\varx\ddd\vart\\
  \xrightarrow[\dt,\dx\to0]{}
  -\int_0^{+\infty}\nspaceint\int_\R \vlim(\vart,\varx)\drondx\monphi(\vart,\varx)\ddd\varx\ddd\vart
  =-\int_0^{+\infty}\nspaceint\int_\R \phi(\ulim(\vart,\varx))\drondx\monphi(\vart,\varx)\ddd\varx\ddd\vart.
  \end{multline*}
  Concerning the third term, it can be written as
  \begin{multline*}
  \int_{0}^{+\infty}\nspaceint\int_\R 
    \udelta(\vart,\varx) \Bigl(
      \frac{\phidelta(\vart,\varx{-}\dx) - \phidelta(\vart,\varx)}{2\dx} \\
      - \frac{\phidelta(\vart,\varx)-\phidelta(\vart,\varx{+}\dx)}{2\dx}
    \Bigr) \ddd\varx\ddd\vart,
  \end{multline*}
  which vanishes as \(\dt,\dx\to0\).
  Let us now treat the right-hand side of \eqref{eq:sc_one_step_con}. After integration by parts, we get
  \begin{multline*}
  s\dt\dx
  \sum_{n\in\N}\sum_{j\in\Z} \bigl(\phi(\ujn)-\vjn\bigr) \frac{\phinjp-\phinjm}{2\dx}\\
  =s\int_{0}^{+\infty}\nspaceint\int_\R
  \bigl(\phi(\udelta)-\vdelta\bigr)\frac{\phidelta(\vart,\varx{+}\dx)-\phidelta(\vart,\varx{-}\dx)}{2\dx}\ddd\varx\ddd\vart.
  \end{multline*}
  Let \(T>0\), \(M>0\), such that \(\operatorname{supp}(\monphi)\subseteq[0,T]\times[-M,M]\). Since 
  \begin{multline*}
  \left\vert 
  \int_{0}^{+\infty}\nspaceint\int_\R
  \bigl(\phi(\udelta)-\vdelta\bigr)\frac{\phidelta(\vart,\varx{+}\dx)-\phidelta(\vart,\varx{-}\dx)}{2\dx}\ddd\varx\ddd\vart
  \right\vert \\
  \leq
  \iint_{\operatorname{supp}(\monphi)}\left\vert \phi(\udelta)-\vdelta\right\vert 
  \norme{\drondx\monphi}_{\infty}\ddd\varx\ddd\vart \\
  \leq
  \norme{\drondx\monphi}_{\infty} \norme{\phi(\udelta)-\vdelta}_{\Linf([0,T];\Lunloc(\R)},
  \end{multline*}
  we conclude, as a consequence of Proposition \ref{th:prop_eq_gap}, that this last term tends towards 0 when \(\dt\) and \(\dx\) go to 0.
  We conclude then that \(\ulim\) is a weak solution of \eqref{eq:s1_edp}.
  
\end{appendix}


\bibliographystyle{amsplain}
\bibliography{Caetano_Dubois_Graille}   

\end{document}